\documentclass[12pt]{article}
\usepackage{amssymb,amscd,amsthm,amsmath,latexsym,amstext,mathrsfs,verbatim}
\usepackage{url} 
\usepackage[all]{xy}
\usepackage{graphicx,enumerate,multicol}

\def\alg{{A}}
\def\np{\bigskip}
\def\op{{\rm op}}
\def\ddA{{\rm A}}

\def\Br{{\rm Br}}
\def\SBr{{\rm SBr}}

\def\BrM{{\rm BrM}}

\def\ddB{{\rm B}}
\def\ddC{{\rm C}}
\def\ddD{{\rm D}}

\def\ddI{{\rm I}}

\def\het{{\rm ht}}
\def\hE{{\hat E}}
\def\hb{{\hat b}}
\def\eps{{\epsilon}}
\def\alp{{\alpha}}

\def\nl{\smallskip\noindent}

\def\lijntje{\vrule height2.4pt depth-2pt width0.5in}

\def\vlijntje{\vrule height0.45in depth0.4pt width0.4pt}
\def\vlijn{\buildrel {\hbox to 0pt{\hss$\textstyle\circ$\hss}}\over\vlijntje}

\def\dlijntje{{\vrule height2pt depth-1.6pt
width0.5in}\llap{\vrule height4pt depth-3.6pt width0.5in}}
\def\vtriple#1\over#2\over#3{\mathrel{\mathop{\kern0pt #2}\limits_{\hbox
to 0pt{\hss$#1$\hss}}^{\hbox to 0pt{\hss$#3$\hss}}}}
\def\rvtriple#1\over#2\over#3{\mathrel{\mathop{\kern0pt #2}\limits_{\hbox
to 0pt{\hss$#3$\hss}}^{\hbox to 0pt{\hss$#1$\hss}}}}

\def\Dn{\vtriple{\scriptstyle1}\over\circ\over{}\kern-1pt\lijntje\kern-1pt
\vtriple{\scriptstyle{2}}\over\circ\over{}
\cdots\cdots\vtriple{\scriptstyle n-3}\over\circ\over{}\kern-1pt\lijntje\kern-1pt
\vtriple{\scriptstyle n-2}\over\circ\over{\buildrel
{\scriptstyle n-1}\over\vlijn}\kern-1pt\lijntje\kern-1pt
\vtriple{n}\over\circ\over{}\kern-1pt}
\def\An{\vtriple{\scriptstyle1}\over\circ\over{}\kern-1pt\lijntje\kern-1pt
\vtriple{\scriptstyle{2}}\over\circ\over{}\kern-1pt\lijntje\kern-1pt
\vtriple{\scriptstyle3}\over\circ\over{}
\cdots\cdots
\vtriple{\scriptstyle n-1}\over\circ\over{}\kern-1pt\lijntje\kern-1pt
\vtriple{\scriptstyle n}\over\circ\over{}\kern-1pt}
\def\Cn{\vtriple{\scriptstyle n-1}\over\circ\over{}
\kern-1pt\lijntje\kern-1pt\vtriple{\scriptstyle{n-2}}\over\circ\over{}
\cdots\cdots
\vtriple{\scriptstyle 2}\over\circ\over{}
\kern-1pt\lijntje\kern-1pt\vtriple{\scriptstyle 1}\over\circ\over{}
\kern-4pt{\dlijntje \kern -25pt<}\kern8pt
\vtriple{\scriptstyle 0}\over\circ\over{}\kern-1pt}

\newcommand{\cA}{\mathcal{A}}
\newcommand{\cB}{\mathcal{B}}

\newcommand{\N}{\mathbb N}

\newcommand{\R}{\mathbb R}
\newcommand{\Z}{\mathbb Z}

\newcommand{\fp}{\mathfrak{p}}

\numberwithin{equation}{section}

\newtheorem{lemma}{Lemma}[section]
\newtheorem{cor}[lemma]{Corollary}
\newtheorem{prop}[lemma]{Proposition}
\newtheorem{thm}[lemma]{Theorem}

\theoremstyle{definition}

\newtheorem{defn}[lemma]{Definition}

\theoremstyle{remark}
\newtheorem{rem}[lemma]{Remark}

\def\a{\alpha}
\def\b{\beta}

\def\alp{\alpha}

\def\isom{\cong}

\begin{document}
\title{Brauer algebras of type C}
\author{Arjeh M.~Cohen, Shoumin Liu, Shona Yu}
\maketitle

\begin{abstract}
For each $n\ge2$, we define an algebra satisfying many properties that one
might expect to hold for a Brauer algebra of type $\ddC_n$. The monomials
of this algebra correspond to scalar multiples of symmetric Brauer diagrams
on $2n$ strands. The algebra is shown to be free of rank the
number of such diagrams and cellular, in the sense of Graham and Lehrer.
\end{abstract}

\medskip\noindent
{\sc keywords:} associative algebra, Birman--Murakami--Wenzl algebra, BMW
algebra, Brauer algebra, cellular algebra, Coxeter group, Tem\-per\-ley--Lieb
algebra, root system, semisimple algebra, word problem in semigroups


\bigskip\noindent
{\sc Corresponding author:} {Arjeh M.~Cohen,
Department of Mathematics and Computer Science,
Eindhoven University of Technology, 
POBox 513, 
5600 MB Eindhoven,
The Netherlands, {email: \verb`A.M.Cohen@tue.nl`}

\section{Introduction}
It is well known that the Coxeter group of type $\ddC_n$ arises from the
Coxeter group of type $\ddA_{2n-1}$ as the subgroup of all elements fixed by
a Coxeter diagram automorphism.  Crisp \cite{Crisp1996} showed that the
Artin group of type $\ddC_n$ arises in a similar fashion from the Artin
group of type $\ddA_{2n-1}$. In this paper, we study the subalgebra of the
Brauer algebra $\Br(\ddA_{2n-1})$ of type $\ddA_{2n-1}$ (that is, the
classical Brauer algebra on $2n$ strands) spanned by Brauer diagrams that
are fixed by the symmetry corresponding to this diagram automorphism (see
Definition \ref{1.1}). Such diagrams will be called symmetric. First, in
Definition \ref{0.1}, we define the Brauer algebra of type $\ddC_n$,
notation
$\Br(\ddC_n)$, in terms
of generators
and relations depending solely on the Dynkin diagram below.
$$ \ddC_n\quad = \quad \Cn$$
The distinguished generators of $\Br(\ddC_n)$ are the involutions
$r_0,\ldots,r_{n-1}$ and the quasi-idempotents $e_0,\ldots,e_{n-1}$ (here, a
quasi-idempotent is an element that is an idempotent up to a scalar
multiple).  Each defining relation concerns at most two indices, say $i$ and
$j$, and is
 determined by the diagram induced by $\ddC_n$ on
$\{i,j\}$. The group algebra of the Coxeter group of type $\ddC_n$ is
obtained by taking
 the quotient of the Brauer algebra of type $\ddC_n$ by
the ideal generated
 by all quasi-idempotents $e_i$. It is isomorphic to
the subalgebra generated by all $r_i$.  The subalgebra generated by all
$e_i$ $(i=0,\ldots,n-1)$ is isomorphic to the Temperley--Lieb algebra
 of
type $\ddB_n$ defined by tom Dieck in~\cite{Dieck2003}.
 
 The main result
states that the algebra $\Br(\ddC_n)$ is isomorphic to the subalgebra
$\SBr(\ddA_{2n-1})$ of the Brauer algebra $\Br(\ddA_{2n-1})$ linearly
spanned by symmetric diagrams.  In order to distinguish them from those of
$\Br(\ddC_n)$, the canonical generators of the Brauer algebra of type
$\ddA_{2n-1}$ are denoted by $R_1,\ldots,R_{2n-1}$, $E_1,\ldots,E_{2n-1}$
instead of the usual lower case letters (see Definition \ref{1.1}).
Although our formal set-up is slightly more general, the algebras considered
are mostly
 defined over the integral group ring $\Z[\delta^{\pm1}]$.
 
\begin{thm}\label{thm:main}
There exists a $\Z[\delta^{\pm1}]$-algebra isomorphism
$$\phi : \Br(\ddC_n)\longrightarrow \SBr(\ddA_{2n-1})$$
determined by $\phi(r_0)=R_n$, $\phi(r_i)=R_{n-i}R_{n+i}$,
$\phi(e_0)=E_n$,  and $\phi(e_i)=E_{n-i}E_{n+i}$, for $0< i< n$.
In particular,
the algebra $\Br(\ddC_n)$ is free over $\Z[\delta^{\pm1}]$
of rank $a_{2n}$,
where $a_n$ is defined by $a_0=a_1=1$ and, for $n>1$, the recursion
$$a_{n} = a_{n-1} + 2(n-1)a_{n-2}.$$
\end{thm}

\noindent A closed formula for the rank of $\Br(\ddC_n)$ is
\begin{eqnarray}
\label{eq:anformula}
a_{2n}&=&\sum_{i =0}^n
    \left(\sum_{p+2q = i}
               \frac{n!}{p! q! (n-i)!}
    \right)^2
     \,2^{n-i}\, (n-i)!.
\end{eqnarray}
A table of $a_n$ for some small $n$ is provided below.

\np

\begin{center}
\begin{tabular}{c|c|c|c|c|c|c|c|c|c}
   $n$& 0 & 1 & 2 & 3 & 4& 5 & 6 & 7 & 8\\
  \hline
  $a_n$& 1 & 1 & 3 & 7 & 25 & 81 & 331 &1303 & 5937 \\
\end{tabular}
\end{center}

\np

The paper is structured as follows. In Section \ref{sect:defns}, we review
the definition of a Brauer monoid of any simply laced Coxeter type and
introduce the notion of a Brauer algebra of type $\ddC_n$, denoted
$\Br(\ddC_n)$.  Section \ref{sect:typeA} reviews facts about the classical
Brauer algebra, denoted $\Br(\ddA_n)$, and about admissible root sets as
presented in \cite{CFW2008}. These sets lead towards a normal form of
monomials closely related to cellularity. In Section
\ref{sect:typeB}, we derive elementary properties of $\Br(\ddC_n)$.  Next,
in Section \ref{sect:surj}, we prove that the image of $\phi$ is precisely
the symmetric diagram subalgebra $\SBr(\ddA_{2n-1})$ of $\Br(\ddA_{2n-1})$.
In Section \ref{sect:adm} we study symmetric diagrams, the related algebra
$\SBr(\ddA_{2n-1})$, and the action of the monoid of all monomials of
$\Br(\ddC_n)$ on certain orthogonal root sets and a normal form for monomials
in $\Br(\ddC_n)$.  With these results, we are able to prove Theorem
\ref{thm:main} in Section \ref{sect:phi}.  Finally, in Section
\ref{sect:cellular}, we establish cellularity (in the sense of \cite{GL1996}) of
the newly introduced Brauer algebras and derive some
further properties.

We finish this introduction by illustrating our results with the first
interesting case: $n=2$.  Consider the classical Brauer algebra
$\Br(\ddA_3)$. The corresponding Brauer diagrams consist of four nodes at the top and
four at the bottom together with a complete matching between these eight
nodes. See Figure \ref{fig:gens_brauerA} for interpretations of $R_i$ and
$E_i$ $(i=1,2,3)$.  A Brauer diagram is called \emph{symmetric} if the complete
matching is not altered by the reflection of the plane whose mirror is the
vertical central axis of the diagram.  Clearly, $e_1:=E_1E_3$, $e_0:=E_2$,
$r_1:=R_1R_3$, and $r_0:=R_2$ represent symmetric diagrams. Our main theorem
implies that the subalgebra of $\Br(\ddA_3)$ generated by these diagrams has
a presentation on these four generators by the relations given in Definition
\ref{0.1} for $n=2$, and moreover it coincides with $\SBr(\ddA_3)$, the linear
span of all symmetric diagrams.  In fact, it is free and spanned by the
following 25 monomials.
\begin{eqnarray*}
1, r_0, r_1, r_0r_1, r_1r_0, r_1r_0r_1, r_0r_1r_0r_1,
r_0r_1r_0, \\
\{1, r_1\}e_0\{1,r_1r_0r_1\}\{1,r_1\},\\
\{1, r_0,  e_0\} e_1 \{ 1,r_0,e_0\}.
\end{eqnarray*}
The first eight, given on the top line, span a subalgebra isomorphic to the group
algebra of the Weyl group of type $\ddC_2$. This is in accordance with the
construction of $\SBr(\ddA_3)$ and the fact that the Weyl group of type
$\ddC_2$ occurs in the Weyl group of type $\ddA_3$ as the subgroup of elements
fixed by a Coxeter diagram automorphism.  The two-sided ideal of $\SBr(\ddA_3)$
generated by $e_1$ is spanned by the $9$ monomials on the bottom line. Also, the
complement in the ideal generated by $e_0$ and $e_1$ of the ideal generated
by $e_1$ is spanned by the $8$ monomials on the middle line. This division of
the 25 spanning monomials into three parts along the above lines is
strongly related to the cellular structure of $\SBr(\ddA_3)$.  The
subalgebra of $\SBr(\ddA_3)$ generated by $e_0$ and $e_1$ has dimension 6
and is isomorphic to the Temperley--Lieb algebra of type $\ddB_2$ introduced
by tom Dieck \cite{Dieck2003}.  For these (and other) reasons, we name
$\SBr(\ddA_3)$ \emph{the Brauer algebra of type} $\ddC_2$. Remarkably, the
Temperley--Lieb algebra of type $\ddB_2$ defined by Graham \cite{Gra} is
7-dimensional and tom Dieck's version is a quotient algebra thereof, but we
have not found a natural extension of Graham's algebra to an object
deserving the name \emph{Brauer algebra of type}~$\ddC_2$.

\section{Definitions} \label{sect:defns}
In this section, we give precise definitions of the algebras and the
homomorphism $\phi$ appearing in Theorem \ref{thm:main}.
All rings and algebras given are unital and associative.

\begin{defn}\label{0.1}
Let $R$ be a commutative ring with invertible element
$\delta$.  For $n\in \N$, the \emph{Brauer algebra of type $\ddC_n$ over $R$
with loop parameter $\delta$}, denoted by $\Br(\ddC_n,R,\delta)$, is the
$R$-algebra generated by $r_0$, $r_1,\dots, r_{n-1}$ and
$e_{0}$, $e_1,\dots, e_{n-1}$ subject to the following relations.
\begin{eqnarray}
r_{i}^{2}&=&1 \qquad \qquad\,\,\,\kern.02em \mbox{for}\,\mbox{any} \ i   \label{0.1.3}
\\
r_ie_i &= & e_ir_i \,=\, e_i \,\,\,\,\,\,\kern.05em \mbox{for}\,\mbox{any}\ i  \label{0.1.4}
\\
e_{i}^{2}&=&\delta^2 e_{i} \qquad \quad\,\,\kern.02em \mbox{for}\ i> 0    \label{0.1.5}
\\
e_{0}^{2}&=&\delta e_{0}             \label{0.1.6}
\\
r_ir_j&=&r_jr_i, \qquad \quad \mbox{for}\ i\nsim j   \label{0.1.7}
\\
e_ir_j&=&r_je_i, \qquad  \quad \kern-.03em \mbox{for}\ i\nsim j     \label{0.1.8}
\\
e_ie_j&=&e_je_i, \qquad \quad \kern-.06em \mbox{for}\ i\nsim j      \label{0.1.9}
\\
r_ir_jr_i&=&r_jr_ir_j, \qquad\,\kern-.04em \mbox{for}\ {i\sim j}\, \mbox{with}\ i, j > 0   \label{0.1.10}
\\
r_jr_ie_j&=&e_ie_j , \quad \qquad \kern-.11em \mbox{for}\ i\sim j\ \mbox{with}\ i,j> 0               \label{0.1.13}
\\
r_ie_jr_i&=&r_je_ir_j , \quad \quad \kern.06em \mbox{for}\ i\sim j\ \mbox{with}\ i, j> 0          \label{0.1.15}
\\
r_{1}r_0r_{1}r_{0}&=&r_0r_{1}r_0r_{1}                                       \label{0.1.11}
 \\
r_{1}r_0e_{1}&=&r_0e_{1}                                \label{0.1.14}
\\
  r_{1}e_0r_{1}e_0&=&e_0e_{1}e_0                                                         \label{0.1.19}
\\
(r_{1}r_0r_{1})e_0&=&e_0(r_{1}r_0r_{1})                                                            \label{0.1.20}
\\
e_{1}r_{0}e_{1}&=&\delta e_{1}                                                \label{0.1.12}
\\
e_{1}e_0e_{1}&=&\delta e_{1}                                                    \label{0.1.16}
\\
e_{1}r_0 r_{1}&=&e_{1}r_{0}                                                          \label{0.1.17}
\\
 e_{1}e_0r_{1}&=&e_{1}e_0                                                            \label{0.1.18}
\end{eqnarray}
Here $i\sim j$ means that $i$ and $j$ are
adjacent in the Dynkin diagram $\ddC_n$. If $R = \Z[\delta^{\pm1}]$ we write
$\Br(\ddC_n)$ instead of $\Br(\ddC_n,R,\delta)$ and speak of {\em the
Brauer algebra of type $\ddC_n$}.  The submonoid of the multiplicative
monoid of $\Br(\ddC_n)$ generated by $\delta$, $\delta^{-1}$,
$\{r_i\}_{i=0}^{n-1}$, and $\{e_i\}_{i=0}^{n-1}$ is denoted by
$\BrM(\ddC_n)$. It is the monoid of monomials in $\Br(\ddC_n)$ and will be
called \emph{the Brauer monoid of type}~$\ddC_n$.
\end{defn}

Observe that, for a distinguished invertible element $\delta$, the ring $R$ can
be viewed as a $\Z[\delta^{\pm1}]$-algebra and that $\Br(\ddC_n,R,\delta)
\isom \Br(\ddC_n)\otimes_{\Z[\delta^{\pm1}]} R$.  As a direct consequence
of the above definition, the submonoid of $\BrM(\ddC_n)$ generated by
$\{r_i\mid i=0,\ldots,n-1\}$ is isomorphic to the Weyl group $W(\ddC_n)$ of
type $\ddC_n$.

Let us recall from \cite{CFW2008} the definition of a
Brauer algebra of simply laced Coxeter type $Q$.
In order to avoid confusion
with the above generators, the symbols of \cite{CFW2008} have been capitalized.

\begin{defn}\label{1.1}
Let $R$ be a commutative ring with invertible element $\delta$ and $Q$
be a simply laced Coxeter graph. The \emph{Brauer algebra of type $Q$ over
$R$ with loop parameter $\delta$}, denoted $\Br(Q,R,\delta)$, is the $R$-algebra generated by $R_i$ and $E_i$, for each node $i$ of $Q$
subject to the following relations, where $\sim$ denotes
adjacency between nodes of $Q$.
\begin{eqnarray}
R_{i}^{2}&=&1          \label{1.1.2} \\
E_{i}^{2}&=&\delta E_{i}   \label{1.1.4} \\
R_iE_i&=&=E_iR_i \,=\, E_i     \label{1.1.3} \\
R_iR_j&=&R_jR_i, \,\, \mbox{for}\, \it{i\nsim j} \label{1.1.5} \\
E_iR_j&=&R_jE_i,\,\, \mbox{for}\, \it{i\nsim j}  \label{1.1.6} \\
E_iE_j&=&E_jE_i,\,\, \mbox{for}\, \it{i\nsim j}    \label{1.1.7} \\
R_iR_jR_i&=&R_jR_iR_j, \,\, \mbox{for}\, \it{i\sim j}  \label{1.1.8} \\
R_jR_iE_j&=&E_iE_j ,\,\, \mbox{for}\, \it{i\sim j}       \label{1.1.9} \\
R_iE_jR_i&=&R_jE_iR_j ,\,\, \mbox{for}\, \it{i\sim j}     \label{1.1.10}
\end{eqnarray}
As before, we call $\Br(Q) := \Br(Q,\Z[\delta^{\pm1}],\delta)$
\emph{the Brauer algebra of type $Q$}
and denote by $\BrM(Q)$
the submonoid of the multiplicative monoid of $\Br(Q)$ generated by $\delta^{-\pm 1}$ and 
all $R_i$ and $E_i$.
 \end{defn}

For any $Q$, the algebra $\Br(Q)$ is free over $\Z[\delta^{\pm 1}]$. Also, the classical Brauer algebra on $m+1$ strands is obtained when $Q=\ddA_m$.

\begin{rem}
As a consequence of the above relations, it is straightforward to show that
the following relations hold in $\Br(Q)$ for all nodes $i$, $j$, $k$ with
$i\sim j \sim k$ and $i\not\sim k$ (see \cite[Lemma 3.1]{CFW2008}).
\begin{eqnarray}
E_iR_jR_j&=&E_iE_j \label{3.1.1} \\ R_jE_iE_j &=& R_i E_j \label{3.1.2} \\
E_iR_jE_i &=& E_i \label{3.1.3} \\ E_jE_iR_j &=& E_j R_i \label{3.1.4} \\
E_iE_jE_i &=& E_i \label{3.1.5} \\ E_jE_iR_k E_j &=& E_jR_iE_k E_j
\label{3.1.6} \\ E_jR_iR_k E_j &=& E_jE_iE_k E_j \label{3.1.7}
\end{eqnarray}
\end{rem}

\begin{rem}
In \cite{Brauer1937}, Brauer gives a diagrammatic description for a basis of
the Brauer algebra of type $\ddA_m$. Each basis element is a diagram with
$2m+2$ dots and $m+1$ strands, where each dot is connected by a unique
strand to another dot.  Here we suppose the $2m+2$ dots have coordinates
$(i, 0)$ and $(i,1)$ in $\R^{2}$ with $1\leq i\leq m+1$. The multiplication
of two diagrams is given by concatenation, where any closed loops formed are
replaced by a factor of $\delta$.  The generators $R_i$ and $E_i$ of $\Br(\ddA_m)$ correspond to the diagrams indicated in Figure \ref{fig:gens_brauerA}.
\end{rem}

\begin{figure}[h!]
\begin{multicols}{2}
  \begin{center}
\hspace*{1cm}
 \includegraphics[width=4.5cm]{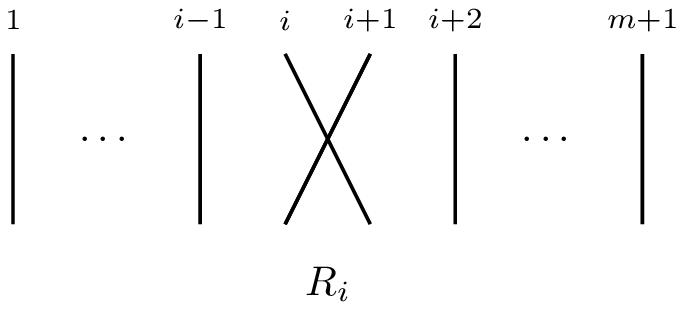} 
      \end{center}
\columnbreak   
  \begin{center}
\hspace*{-1.2cm}
\includegraphics[width=4.5cm]{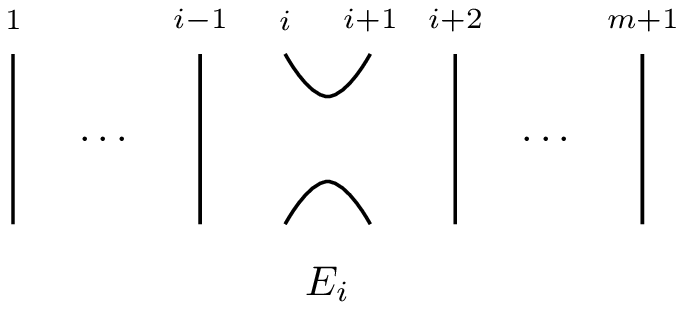}
  \end{center}
\end{multicols}
\caption{Brauer diagrams corresponding to $R_i$ and $E_i$.}
\label{fig:gens_brauerA}
\end{figure}

Each Brauer diagram can be written as a product of elements from
$\{R_i,E_i\}_{i=1}^m$. This statement is illustrated in Figure
\ref{fig:BrauerDiag}.

\begin{figure}[h!]
\begin{center}
   \includegraphics[scale=0.55]{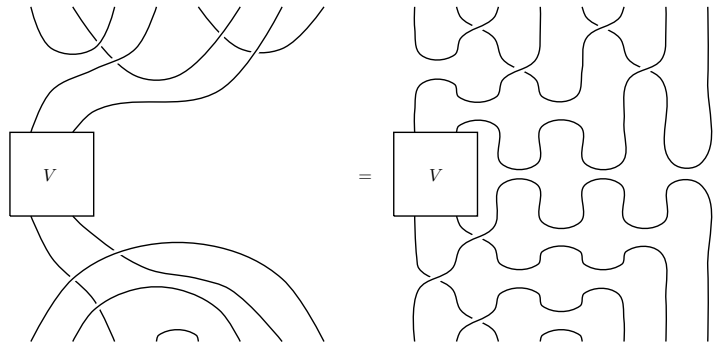}
\vspace{-9cm}
\caption{A Brauer diagram and a visualization of it as the product
$R_2R_5E_1R_3R_6E_2E_4VE_3E_5E_7R_2E_4E_6R_1E_3E_5R_2E_4$, where $V$ can be
  either the identity or the simple crossing $R_1$.}
\label{fig:BrauerDiag}
\end{center}
\end{figure}

Henceforth, we identify $\BrM(\ddA_{m})$ with its diagrammatic version. It
makes clear that $\Br(\ddA_m)$ is a free algebra over $\Z[\delta^{\pm1}]$ of
rank $(m+1)!!$, the product of the first $m+1$ odd integers. The monomials
of $\BrM(\ddA_m)$ that correspond to diagrams
will be referred to as diagrams.

The map $\sigma$ on the graph (or
Coxeter type) $\ddA_{m}$ given by $\sigma(i) = m+1-i$ is the single
nontrivial automorphism of this graph. As the presentation of $\Br(Q)$
merely depends on the graph $Q$, the map $\sigma$ induces an automorphism of
$\Br(\ddA_m)$, which will also be denoted by $\sigma$.  This involutory
automorphism is determined by its behaviour on the generators:
\[\sigma(R_{i})=R_{m+1-i},\,\,\,\,\sigma(E_{i})=E_{m+1-i}.\]

The automorphism $\sigma$ may be viewed simply as a reflection of the corresponding diagram about its central vertical axis.

\begin{defn}
Suppose $D_1$ and $D_2$ are diagrams in $\Br(\ddA_{m})$.
The diagram $D_1$ is \emph{symmetric} to the diagram $D_2$ if $D_2$ is the
diagram obtained by taking the reflection of $D_1$ about its central
vertical axis. If $D_1=D_2$, then we say $D_1$ is a symmetric diagram.
\end{defn}

\noindent Hence a diagram (that is, a monomial in $R_i$ and $E_i$) of
$\Br(\ddA_{m})$ is $\sigma$-~invariant if and only if it is symmetric about its
central vertical axis.
A monomial in $\BrM(\ddA_{m})$ is fixed by $\sigma$ if and only if it
represents a symmetric diagram.

\begin{lemma}
\label{3.4}
Let $m=2n-1$, for some $n \in \N$.  The number of symmetric diagrams (with
respect to $\sigma$) in $\BrM(\ddA_m)$ is equal to $a_{2n}$, where $a_n$
satisfies $a_0=a_1=1$ and the recursion
$$a_{n} = a_{n-1} + 2(n-1)a_{n-2}.$$
\end{lemma}

\begin{proof} Fix two sets $X$ and $Y$, say, of size $n$ and a
permutation $\tau $ of $X\cup Y$ of order 2 interchanging $X$ and $Y$.  We
define $a_n$ as the number of perfect matchings on $X\cup Y$ that are
$\tau$-invariant (that is, if $\{a,b\}\subseteq X\cup Y$ belongs to the
matching, so does $\{\tau(a),\tau(b)\}$). Identifying $X$ with the set of
dots left of the vertical axis of symmetry, $Y$ with the set of dots to the
right, and $\tau$ with the permutation induced by $\sigma$, we see that
that $a_{2n}$ is the number of symmetric diagrams in $\BrM(\ddA_m)$.

It is obvious that $a_0=a_1=1$.  Fix $a\in X$.  The number of perfect
$\tau$-invariant matchings containing $\{a,\tau(a)\}$ is equal to $a_{n-1}$.

Suppose that we have a perfect $\tau$-invariant matching of $X\cup Y$
containing $\{a,b\}$ with $b\ne\tau(a)$.  Then $\{\tau(a),\tau(b)\}$ is a
second pair belonging to the matching. The matching induces $a_{n-2}$ number of perfect
$\tau$-invariant matchings on $(X\cup Y)\setminus \{a,b,\tau(a),\tau(b)\}$.  As there are $2n-2$ choices of $b$, we find
$a_{n}=a_{n-1}+(2n-2)a_{n-2}$.
\end{proof}

\begin{cor}\label{cor:3.4}
The linear span $\SBr(\ddA_{2n-1})$ of symmetric diagrams
is a $\Z[\delta^{\pm1}]$-subalgebra of
$\Br(\ddA_{2n-1})$. It is free over $\Z[\delta^{\pm1}]$ of rank $a_{2n}$.
\end{cor}

Observe that $R_n$, $R_iR_{2n-i}$, $E_n$, and $E_iE_{2n-i}$ are
fixed under $\sigma$ for all $i\in\{1,\ldots,n\}$. Thus the image of the map
$\phi$ of Lemma \ref{lm:phiIsHomo} below lies in $\SBr(\ddA_{2n-1})$.

\begin{lemma}
\label{lm:phiIsHomo}
The following map determines a $\Z[\delta^{\pm 1}]$--algebra homomorphism
$\phi:\Br(\ddC_n)\to\SBr(\ddA_{2n-1})$.
\begin{center}
 $\phi(r_0)=R_n$, \, $\phi(r_i)=R_{n-i}R_{n+i}$, \\
$\phi(e_0)=E_n$\, and \,$\phi(e_i)=E_{n-i}E_{n+i}$, for $0< i< n$.
\end{center}
\end{lemma}

\begin{proof}
It suffices to verify that $\phi$ preserves the defining relations given in Definition
\ref{0.1}.
We demonstrate this for some of the relations (\ref{0.1.3})--(\ref{0.1.18}),
and leave the
rest as an exercise for the reader.

\noindent For (\ref{0.1.19}):
\begin{eqnarray*}
\phi(r_1)\phi(e_0)\phi(r_1)\phi(e_0)
&=&R_{n-1}(R_{n+1}E_n R_{n+1})R_{n-1}E_n\\
&\overset{(\ref{1.1.10})}{=}&R_{n-1}R_{n}E_{n+1}(R_{n}R_{n-1}E_{n})\\
&\overset{(\ref{1.1.9})}{=}&R_{n-1}R_{n}E_{n+1}E_{n-1}E_{n}\\
&\overset{(\ref{1.1.7})+(\ref{1.1.9})}{=}&E_nE_{n-1}E_{n+1}E_{n}\\
&=&\phi(e_0)\phi(e_1)\phi(e_0).
\end{eqnarray*}

\noindent For (\ref{0.1.20}):
\begin{eqnarray*}
\phi(r_1)\phi(r_0)\phi(r_1)\phi(e_0)
&=&R_{n+1}(R_{n-1}R_{n}R_{n-1})R_{n+1}E_{n}\\
&\overset{(\ref{1.1.8})}{=}&R_{n+1}R_nR_{n-1}(R_{n}R_{n+1}E_{n})\\
&\overset{(\ref{1.1.9})+(\ref{3.1.1})}{=}&R_{n+1}R_nR_{n-1}E_{n+1}R_n R_{n+1}\\
&\overset{(\ref{1.1.6})}{=}&R_{n+1} R_{n}E_{n+1}R_{n-1}R_n R_{n+1}\\
&\overset{(\ref{1.1.9})+(\ref{3.1.1})}{=}&E_nR_{n+1}R_n R_{n-1}R_n R_{n+1}\\
&\overset{(\ref{1.1.8})}{=}&E_{n}R_{n+1}R_{n-1}R_{n}R_{n-1}R_{n+1}\\
&=&\phi(e_0)\phi(r_1)\phi(r_0)\phi(r_1).
\end{eqnarray*}

\noindent For (\ref{0.1.18}):
\begin{eqnarray*}
\phi(e_1)\phi(e_0)\phi(r_1)
&\overset{(\ref{1.1.5})}{=}&E_{n-1}(E_{n+1}E_{n}R_{n+1})R_{n-1}
\overset{(\ref{3.1.4})}{=}E_{n-1}E_{n+1}R_nR_{n-1}\\
&\overset{(\ref{1.1.7})+(\ref{3.1.1})}{=}&
E_{n+1}E_{n-1}E_{n}=\phi(e_1)\phi(e_0).
\end{eqnarray*}
\end{proof}

At this point, we have explained the algebras and the map $\phi$ occurring
in Theorem \ref{thm:main}.  The surjectivity of $\phi$ will be proved in
Proposition \ref{prop:phisurj} and its injectivity at the end of Section
\ref{sect:phi}.

\section{The classical Brauer algebra} \label{sect:typeA}
Let $m\in\N$.  In this section, we describe the root system of the Coxeter
group of type $\ddA_m$, focussing on special collections of mutually
orthogonal positive roots called admissible sets.  Also, the notion of
height for elements of the Brauer monoid $\BrM(\ddA_{m})$ is introduced and
discussed. A major goal, established in Theorem \ref{th:AmDecomp}, is to
exhibit a normal form for elements of the monoid $\BrM(\ddA_m)$ as a
product of generators.

\begin{defn}\label{defn:Phi}
Let $m\ge1$.  The root system of the Coxeter group $W(\ddA_{m})$ of type
$\ddA_{m}$ is denoted by $\Phi$.  It is realized as $\Phi :=
\{\eps_i-\eps_j\mid 1\le i,j\le m+1,\ i\ne j\}$ in the Euclidean space
$\R^{m+1}$, where $\eps_i$ is the $i^\mathrm{th}$ standard basis vector.  Put $\alp_i
:= \eps_i-\eps_{i+1}$.  Then $\{\alpha_{i}\}_{i=1}^{m}$ is called the set of
simple roots of $\Phi$.  Denote by $\Phi^+$ the set of positive roots in
$\Phi$ with respect to these simple roots; that is, $\Phi^+ := \{\eps_i-\eps_j\mid 1\le i<j\le m+1\}$.
\end{defn}

We have seen that, up to powers of $\delta$, the monomials of $\Br(\ddA_m)$
correspond to Brauer diagrams. In order to work with the tops and bottoms of
Brauer diagrams, we introduce the following notion.

\begin{defn}
\label{df:cA}
Let $\cA$ denote the collection of all subsets of $\Phi$ consisting of
mutually orthogonal positive roots.
Members of $\cA$ are called \emph{admissible sets}.
\end{defn}

An admissible set $B$ corresponds to a Brauer diagram top in the following
way: for each $\b\in B$, where $\b = \eps_i-\eps_j$ for some
$i,j\in\{1,\ldots,m+1\}$ with $i<j$, draw a horizontal strand in the
corresponding Brauer diagram top from the dot $(i,1)$ to the dot
$(j,1)$. All horizontal strands on the top are obtained this way, so there
are precisely $|B|$ horizontal strands.  The top of the Brauer diagram of
Figure \ref{fig:BrauerDiag}, corresponds to the admissible set
$\{\alpha_1+\alpha_2, \alpha_2+\alpha_3+\alpha_4+\alpha_5,
\alpha_5+\alpha_6+\alpha_7\}$.

Similarly, there is an admissible set corresponding to a Brauer diagram
bottom.  The bottom of the Brauer diagram of Figure \ref{fig:BrauerDiag},
corresponds to the admissible set
$\{\alpha_1+\alpha_2+\a_3+\a_4+\a_5+\a_6+\a_7,
\alpha_2+\alpha_3+\alpha_4+\a_5, \alpha_4\}$.

For any $\beta\in\Phi^+$ and $i\in\{1,\ldots,m\}$, there exists a $w\in
W(\ddA_m)$ such that $\beta = w\alpha_i$. Then $E_\beta := wE_iw^{-1}$ is
well defined (see \cite[Lemma 4.2]{CFW2008}).  If $\beta,\gamma\in\Phi^+$
are mutually orthogonal, then $E_\beta$ and $E_\gamma$ commute (see
\cite[Lemma 4.3]{CFW2008}). Hence, for $B\in\cA$, we can define
the product
\begin{eqnarray}
\label{eqn:EprodB}
E_B &=& \prod_{\beta\in B} E_\beta,
\end{eqnarray}
which is a quasi-idempotent, and the normalized version
\begin{eqnarray}
\label{eqn:EhatB}
\hE_B &=& \delta^{-|B|} E_B,
\end{eqnarray}
which is an idempotent element of the Brauer monoid.

There is an action of the Brauer monoid
$\BrM(\ddA_{m})$ on the collection $\cA$.
The generators $R_{i}$ $(i=1,\ldots,m)$ act
by the natural action of Coxeter group elements on its
root sets, where negative roots are negated so as to obtain positive roots,
the element $\delta$ acts as the identity,
and the action of $E_{i}$ $(i=1,\ldots,m)$ is defined by
\begin{equation}
E_i B :=\begin{cases}
B & \text{if}\ \alpha_i\in B, \\
B\cup \{\alpha_{i}\} & \text{if}\ \alpha_i\perp B,\\
R_\beta R_i B & \text{if}\ \beta\in B\setminus \alpha_{i}^{\perp}.
\end{cases}
\end{equation}

Alternatively, this action can be described as follows for a monomial $a$:
complete the top corresponding to $B$ into a Brauer diagram $b$, without
increasing the number of horizontal strands in the top. Now $aB$ is the top
of the Brauer diagram $ab$.  We will make use of this action in order to
provide a normal form for elements of $\BrM(\ddA_m)$.
In \cite[Definition 3.2]{CFW2008}, it is shown that this action is well
defined for any spherical simply laced type.

The action defined above is a left action. Similarly, there is a right
action of $\BrM(\ddA_{m})$ on $\cA$. In order to interpret the right action
of $a$ on $B$, the latter should be pictured as the bottom of a Brauer
diagram, so that the bottom corresponding to $Ba$ is the bottom of $ba$.
Observe that $a\emptyset$ is the top of the Brauer diagram of $a$ (i.e., the
collection of top horizontal strands of $a$ and top row of points) and
$\emptyset a$ is its bottom (i.e., the collection of bottom horizontal strands
of $a$ and bottom row of points).

Recall that the height $\het(\b)$ of a positive root $\beta\in\Phi^+$ is $h$
if it is the sum of precisely $h$ simple roots.  This definition will be
extended to a height function on $\cA$ in such a way that $\het(\{\b\}) =
\het(\b)-1$.

\begin{defn}
\label{df:hetB}
The height of an admissible set $B$, notation $\het(B)$, is
the minimal number of crossings
in a completion of the top corresponding to $B$
to a Brauer diagram without increasing the number of
horizontal strands at the top.
\end{defn}

For example, the height of the top of the Brauer diagram of Figure
\ref{fig:BrauerDiag} is equal to 4 and the height of the bottom is equal to
3.

\begin{defn}
\label{df:hetA}
For every element $a\in \BrM(\ddA_{m})$, we define the height of $a$,
denoted by $\het(a)$, as the minimal number of generators $R_i$ needed to
write $a$ as a product of the generators $R_1,\ldots,R_m, E_1,\ldots,E_m$,
$\delta$, $\delta^{-1}$.
\end{defn}

In terms of Brauer diagrams, the height of $a$ is the minimal number of
crossings needed to draw $a$. Consequently, the height of an admissible set
$B$ is the minimal height over all possible Brauer diagram completions of
$B$.

The Brauer diagram of Figure \ref{fig:BrauerDiag} has height 7 if $V$ is the
identity and height 8 if $V=R_1$.  The lemma below states some useful
properties of this height function.

\begin{rem}\label{rem:oppAn}
There is a natural anti-involution on $\Br(\ddA_m)$,
denoted by $x\mapsto x^{\op}$, determined by
$$R_i \mapsto\ R_i \text{ and } E_i \mapsto E_i.$$
By anti-involution, we mean a $\Z[\delta^{\pm 1}]$-linear anti-automorphism
whose square is the identity.
\end{rem}

\begin{lemma}
\label{lm:hetB}
Let $B,C\in\cA$ with $|B|=|C|$.
Then
there is a unique diagram $a_{B,C}$ of height $\het(B)+\het(C)$
in $\BrM(\ddA_m)$ such that $a_{B,C}\emptyset = B$ and $\emptyset a_{B,C} =
C$.  This diagram satisfies $a_{B,C}a_{B,C}^{\op} =
\delta^{|B|}E_{B}$ as well as $a_{B,C}C = B$ and $B a_{B,C} = C$.
\end{lemma}

\begin{proof}
The easiest proof to our knowledge is based on diagrams.

There is a unique way to complete a
given top and bottom to a Brauer diagram with a minimal number of crossings:
connect the first dot at the top from the left that is not the endpoint of a
horizontal strand at the top to the first dot at the bottom that is not the
endpoint of a horizontal strand at the bottom; proceed similarly with the
second, and so on, until the Brauer diagram is complete. If the top
corresponds to $B$ and the bottom to $C$, the resulting
diagram is the required monomial $a_{B,C}$.
\end{proof}

\begin{lemma}\label{lm:heightZero}
Suppose $B\in\cA$ has height $0$.  Then there are $r = m-2|B|$ diagrams of
height 1 in the group of invertible elements in $\hE_B\BrM(\ddA_m)\hE_B$
forming a Coxeter system of type $\ddA_r$. (Here, invertibility is meant with respect to the unit
$\hE_B$ of the monoid).
\end{lemma}

\begin{proof}
As discussed above, a Brauer diagram with top and bottom corresponding to
$B$ has $r+1$ free dots at the top and also $r+1$ at the bottom. For a
diagram to be an invertible element in $\hE_B\BrM(\ddA_m)\hE_B$, the
remaining strands need to be vertical, so they belong to the symmetric group on
the $r+1$ free dots at the top (or those at the bottom). Now, up to powers
of $\delta$, the idempotent $\hE_B$ is the element in which all $r$ vertical
strands do not cross. Selecting diagrams in which the $i^\mathrm{th}$ and $(i+1)^\mathrm{st}$
vertical strands cross and no others (for $i=1,2,\ldots,r$), we find the
required Coxeter system of type $\ddA_r$.
\end{proof}

\begin{defn}\label{df:KB}
For $B\in\cA$ of height $0$, denote by $K_B$ the Coxeter group
determined by Lemma \ref{lm:heightZero}.
\end{defn}

In the middle part of the right hand side of
Figure \ref{fig:BrauerDiag}, next to $V$,
the element $e_3e_5e_7=E_B$ appears,
where $B = \{\a_3,\a_5,\a_7\}$ has height $0$.
Now $K_B$ is a Coxeter group of type $\ddA_1$,
generated by $R_1\hE_B$, so the choices for $V$ are consistent
with the  possibilities for $V\hE_B=\hE_B V \hE_B\in K_B$.

\begin{thm} \label{th:AmDecomp}
Let $i\in\{0,1,\ldots,\lfloor m/2\rfloor\}$ and let $B$ be any admissible set
of size $i$ and of height $0$.
Then each element $a$ of $\BrM(\ddA_{m})$ with $|a\emptyset|=i$
can be written uniquely
as
$$\delta^k UVW$$
for certain $k\in \Z$,
$U$ a diagram in $\BrM(\ddA_{m})E_{B}$ with $UB = a\emptyset$,
$W$ a diagram in $ E_{B} \BrM(\ddA_{m})$ with $\emptyset a = BW$,
and $V\in K_{B}$
such that $$\het(a) = \het(U)+\het(V)+\het(W) .$$
\end{thm}

\begin{proof}
Take $U =a_{a\emptyset,B}$ and $W = a_{\emptyset a, B}^{\op}$.
Then $V = U^{\op}a W^{\op}$ has top and bottom equal to $B$ and so belongs to
$K_B$.
By Lemma \ref{lm:hetB} and (\ref{eqn:EhatB}),
\begin{eqnarray*}
UVW &=& a_{a\emptyset,B} a_{a\emptyset,B}^{\op} a
a_{\emptyset a, B} a_{\emptyset a, B}^{\op}
=\delta^{4i} \hE_{a\emptyset} a \hE_{\emptyset
  a} = \delta^k a
\end{eqnarray*}
for $k = 4i$.
As $\het(U) = \het(a\emptyset)$ and
$\het(W) = \het(\emptyset a)$ and $\het(V)$ is the length of $V$ with
respect to the Coxeter system of Lemma \ref{lm:heightZero},
this proves that $a$ has a decomposition as stated.

As for uniqueness, suppose $a=UVW$ is a product decomposition as stated.
Then $a\emptyset = UVB=UB=U \emptyset$ (as $U = U\hE_B$)
and $\emptyset U = B$, so by Lemma
\ref{lm:heightZero} and minimality of the height of $U$, we find $U =
a_{a\emptyset,B}$. Similarly, $W$ can be shown to be equal to $ a_{\emptyset
a,B}^{\op}$.  Finally, $V = \hE_BV\hE_B =\delta^{-4|B|}
 U^{\op} UV W W^{\op} = \delta^{-4|B|} U^{\op} a
W^{\op}$ is uniquely determined by $a$, $U$, and $W$.
\end{proof}

\section{Elementary properties of type $\ddC$ algebras} \label{sect:typeB}
In this section we draw some easy consequences from the definition of a
Brauer algebra of type $\ddC_n$. The results are of use in later sections.

\begin{lemma} In $\Br(\ddC_n,R,\delta)$, the following
equations hold.
\begin{eqnarray}
r_{1}e_0 e_{1}&=&e_{0}e_{1}     \label{4.1.2}
\\
e_0e_{1}e_{0}&=&e_0r_{1}e_0  \label{4.1.1}
\\
e_{1}r_{0}r_{1}e_{0}&=&e_{1}e_{0}       \label{4.1.3}
\\
r_0r_{1}e_{0}r_{1}&=&r_{1}e_{0}r_{1}r_0       \label{4.1.4}
\\
e_0r_{1}e_0r_{1}&=&e_0e_{1}e_0                \label{4.1.5}
\end{eqnarray}
\end{lemma}

\begin{proof} By Definition \ref{0.1},
\[
r_{1}e_0 e_{1} \overset{(\ref{0.1.16})}{=} \delta^{-1}r_{1}e_0e_{1}e_0 e_{1} \overset{(\ref{0.1.19})}{=} \delta^{-1} e_0 e_{1}e_0 e_{1} \overset{(\ref{0.1.16})}{=} e_0 e_{1},
\]
proving (\ref{4.1.2}).
Therefore
\[ e_0e_{1}e_{0} \overset{(\ref{4.1.2})}{=} r_{1}e_{0}e_{1} e_{0}
\overset{(\ref{0.1.19})}{=} r_{1}r_{1} e_0 r_{1} e_0
\overset{(\ref{0.1.3})}{=} e_0r_{1}e_{0}\]
giving (\ref{4.1.1}).

It is easy to check that (\ref{4.1.3}) follows from (\ref{0.1.17}) and
(\ref{0.1.4}). Also, the identity (\ref{4.1.4}) holds as
  \begin{eqnarray*}
 r_0r_{1}e_{0}r_{1}
 &\overset{(\ref{0.1.3})}{=}&r_{1}(r_{1}r_0r_{1}e_{0})r_{1}
\overset{(\ref{0.1.20})}{=}r_{1}e_{0}r_{1}r_0r_{1}r_{1}
\overset{(\ref{0.1.3})}{=}r_{1}e_{0}r_{1}r_0.
  \end{eqnarray*}

\noindent Finally,
\[
(e_0r_{1}e_0)r_{1} \overset{(\ref{4.1.1})}{=} e_0(e_{1}e_{0}r_{1})
\overset{(\ref{0.1.18})}{=} e_0e_{1}e_{0},
\]
proving (\ref{4.1.5}).
\end{proof}

As in the case of type $\ddA_m$ (see Remark \ref{rem:oppAn})
there is similarly a natural anti-involution on $\Br(\ddC_n)$. This anti-involution is denoted by the superscript $\op$, so the map is given by $x\mapsto x^{\op}$.

\begin{prop} \label{prop:opp}
The identity map on
$\{\delta, r_i, e_i \mid i = 0,\ldots,n-1\}$ extends to a unique
anti-involution on the Brauer algebra $\Br(\ddC_n,R,\delta)$.
\end{prop}

\begin{proof}
It suffices to check the defining relations given in Definition \ref{1.1}
still hold under the anti-involution.  An easy inspection shows that all
relations involved in the definition are invariant under $\op$, except for
(\ref{0.1.13}), (\ref{0.1.14}), (\ref{0.1.19}), (\ref{0.1.17}), and
(\ref{0.1.18}).  The relation obtained by applying $\op$ to (\ref{0.1.13})
holds as can be seen by using (\ref{0.1.15}) followed by (\ref{0.1.3})
together with (\ref{0.1.13}).  The equality \ref{0.1.17} is the op-dual of
(\ref{0.1.14}). Finally, (\ref{4.1.2}) and (\ref{4.1.5}) state that the $\op$
duals of (\ref{0.1.18}) and (\ref{0.1.19}), respectively, hold.
\end{proof}

\noindent For each $i\in\{1,\ldots,n\}$, we define the following two elements of $\BrM(\ddC_n)$.
\begin{eqnarray}
\label{df:yn}
y_i&:=&r_{i-1}r_{i-2}\cdots r_1r_0r_1\cdots r_{i-2}r_{i-1},\\
z_i&:=&r_{i-1}r_{i-2}\cdots r_1e_0r_1\cdots r_{i-2}r_{i-1}.
\label{df:zn}
\end{eqnarray}

\begin{prop}\label{X_n}
Let $n\ge2$ and $i\in\{2,\ldots,n\}$ and consider elements in
$\BrM(\ddC_n)$.
\begin{enumerate}[(i)]
\item  $e_{i}, r_{i},y_i$,
and $z_i$ commute with each of
$r_j$ and $e_j$ for $0\le j \le i-2$.
\item $y_i$ and $z_i$ commute with $y_j$ and $z_j$
for each $j\in\{1,\ldots,n\}$.
\end{enumerate}
\end{prop}

\begin{proof}
(i).
By Definition \ref{0.1}, both $e_{i}$ and $r_{i}$ commute
with each element of $\{r_0,\ldots,r_{i-2},e_0,\ldots,e_{i-2}\}$.

In order to prove that
$y_i$ commutes with the indicated elements, we first
establish the claim that $y_{i+2}$ commutes with $r_i$ and $e_i$, for $0
\leq i \leq n-2$.

If $i = 0$, the claim follows from (\ref{0.1.11}) and (\ref{0.1.20}),
respectively. If $i > 0$, we have
\begin{eqnarray*}
  y_{i+2} r_i &=& r_{i+1} r_i \cdots r_1 r_0 r_1 \cdots r_i r_{i+1} r_i
\overset{(\ref{0.1.10})}{=} r_{i+1} r_i \cdots r_1 r_0 r_1 \cdots r_{i-1} r_{i+1} r_i r_{i+1} \\
  &\overset{(\ref{0.1.7})}{=}&r_{i+1} r_i r_{i+1} r_{i-1} \cdots r_1 r_0 r_1 \cdots r_i r_{i+1} \\
&\overset{(\ref{0.1.10})}{=}&
 r_i r_{i+1} r_i \cdots r_1 r_0 r_1 \cdots r_{i-1} r_i r_{i+1} = r_i y_{i+2},
\end{eqnarray*}
and
\begin{eqnarray*}
  y_{i+2} e_i &=& r_{i+1} r_i \cdots r_1 r_0 r_1 \cdots r_i r_{i+1} e_i
\overset{(\ref{0.1.15})}{=} r_{i+1} r_i \cdots r_1 r_0 r_1 \cdots r_{i-1} e_{i+1} r_i r_{i+1} \\
  &\overset{(\ref{0.1.8})}{=}&r_{i+1} r_i e_{i+1} r_{i-1} \cdots r_1 r_0 r_1
\cdots r_i r_{i+1}
\\
&\overset{(\ref{0.1.15})}{=}&
e_i r_{i+1} r_i \cdots r_1 r_0 r_1 \cdots r_{i-1} r_i r_{i+1} = e_i y_{i+2}.
\end{eqnarray*}
Now, for arbitrary $i$ and all $0 \leq j \leq i-2$,
using
$y_i = r_{i-1}\cdots r_{j+2} y_{j+2} r_{j+2} \cdots  r_{i-1}$
and  (\ref{0.1.7}), we find
\begin{eqnarray*}
y_ir_j &=&  r_{i-1}\cdots r_{j+2} y_{j+2} r_{j+2} \cdots  r_{i-1}r_j
=r_{i-1}\cdots r_{j+2} y_{j+2}r_j r_{j+2} \cdots  r_{i-1}\\
&=&r_{i-1}\cdots r_{j+2}r_j y_{j+2} r_{j+2} \cdots  r_{i-1}
=r_jr_{i-1}\cdots r_{j+2} y_{j+2} r_{j+2} \cdots  r_{i-1}\\
&=&r_jy_i
\end{eqnarray*}
Similarly for $e_j$ instead of $r_j$, by use of (\ref{0.1.8}).

An analogous argument can be used for $z_i$.
Again, it suffices to show that $z_{i+2}$ commutes with $r_i$ and $e_i$.
For $i=0$ we verify
\begin{eqnarray*}
z_2 e_0 &=& (r_1 e_0 r_1) e_0 \overset{(\ref{0.1.19}) + (\ref{4.1.5})}{=} e_0 (r_1 e_0 r_1) = e_0 z_2,
\\
z_2 r_0 &=&
(r_1 e_0 r_1) r_0 \overset{(\ref{0.1.3})}{=} r_1 e_0 r_1 r_0 r_1 r_1 \overset{(\ref{0.1.20}) + (\ref{0.1.3})}{=} r_0 (r_1 e_0 r_1) = e_0 z_2.
\end{eqnarray*}
For $i > 0$, it is straightforward to show that $z_{i+2} r_i = r_i z_{i+2}$, using (\ref{0.1.7}) and (\ref{0.1.10}), and $z_{i+2} e_i = e_i z_{i+2}$, by using a rearrangement of (\ref{0.1.15}).
Thus, an argument  similar to the above proves that $z_{i}r_{j} = r_j z_i$
and $z_{i}e_{j} = e_j z_i$, for any $i$ and all $0 \leq j \leq i-2$.

As $y_i$ and $z_i$ are conjugates of $r_0$ and $e_0$ by the same
Coxeter group element, it follows from (\ref{0.1.4}) that they commute.  It
remains to verify that $y_i$ and $z_i$ commute with $y_{i-1}$ and $z_{i-1}$.
But the latter two elements are products of generators from $\{r_0,\ldots,
r_{i-2},e_0,\ldots,e_{i-2}\}$, which are known to commute with $y_i$ and $z_i$
by (i) and (ii). This finishes the proof of (i) and (ii).
\end{proof}

\begin{rem}
Using Proposition \ref{X_n}, one can prove that for any $n \geq 1$,
$$\BrM(\ddC_n)=
\BrM(\ddC_{n-1})\{1, e_{n-1}, r_{n-1},y_n, z_n \}\BrM(\ddC_{n-1}).$$
This ensures that $\Br(\ddC_n)$ is of \emph{finite rank} over
$\Z[\delta^{\pm 1}]$. Details of the proof of the ensuing `normal form' are
surpressed here, as this result is not used in this paper.
\end{rem}

\section{Surjectivity of $\phi$} \label{sect:surj}
The goal of this section is to exhibit a collection of admissible sets on
which $\BrM(\ddC_n)$ acts as well as to prove that the map $\phi :
\Br(\ddC_n)\longrightarrow \SBr(\ddA_{2n-1})$ introduced in Theorem
\ref{thm:main} is surjective.  To this end, we first construct the root
system of type $\ddC_n$ in terms of $\sigma$-fixed vectors in the reflection
space for $W(\ddA_{2n-1})$ spanned by the root system $\Phi$ of Definition
\ref{defn:Phi}.  Notice that the restriction of $\phi$ to the submonoid
$W(\ddC_n)$ of $\BrM(\ddC_n)$ generated by the $r_i$ (isomorphic, as the
notation suggests, to the Coxeter group of type $\ddC_n$) is known to be
injective (see, for instance \cite{Muehl92}), with image the centralizer of
$\sigma$ in the submonoid $W(\ddA_{2n-1})$ of $\BrM(\ddA_{2n-1})$.

We adopt the notation of Section \ref{sect:typeA} with $m=2n-1$, and will us
the root system $\Phi$, the collection of admissible sets $\cA$, and the
action of $\BrM(\ddA_{2n-1})$ on $\cA$ defined there.
We let the involution $\sigma$ act on the set $\Phi^+$ of positive roots of
$\Phi$ in the following way,
where $\alp_i$ are as in Definition
\ref{defn:Phi}.
For $\Sigma c_{i}\alpha_{i}\in
\Phi^{+}$ we decree
$$\sigma(\Sigma c_{i}\alpha_{i})=\Sigma c_{i}\alpha_{2n-i} \qquad(1\le
 i<2n).$$ This map $\sigma$ induces the permutation of the simple roots
corresponding to the nontrivial automorphism of the Coxeter diagram
 $\ddA_{2n-1}$.
This permutation can be extended to the
linear transformation of $\R^{2n}$, again denoted $\sigma$,
determined by $\sigma(\eps_i) =
- \eps_{2n+1-i}$.
(Indeed, this transformation satisfies $\sigma(\alp_i) = \alp_{2n-i}$ for each $i\in\{1,\ldots,2n-1\}$).
The vectors fixed by $\sigma$ form an $n$-dimensional subspace, to be
denoted
$\R^{2n}_\sigma$, of
the $(2n-1)$-dimensional subspace of $\R^{2n}$
spanned by $\Phi$.

\begin{defn}\label{df:fp}
Let $\fp :\R^{2n}\to\R^{2n}_\sigma$ be the orthogonal projection from
$\R^{2n}$ onto $\R^{2n}_\sigma$, that is, $\fp(x) = (x+\sigma(x))/2$ for
$x\in\R^{2n}$.  Let $\alp\in\Phi$. Then $\fp(\alp) = \alpha$ is of squared
norm 2 if $\sigma(\alp) = \alp$ and $\fp(\alp) =
\frac{1}{2}(\alpha+\sigma(\alpha))$ is of squared norm $1$ if $\sigma(\alp)
\ne \alp$.  The image $\Psi = \fp(\Phi)$ of $\Phi$ under $\fp$ is a root
system of type $\ddC_n$ with simple roots $\beta_0=\fp(\alpha_n)=\a_n$ and
$\beta_i=\fp(\alpha_{n-i})=\fp(\alpha_{n+i})$ for $i=1,\dots,n-1$.  It is
contained in $\R^{2n}_\sigma$ and spans it. Of course, $\Psi^+$ will be
understood to be the half of $\Psi$ lying in the cone spanned by $\beta_i$
$(i=0,\ldots,n-1)$. Given $\alp\in\Phi$ we write $R_\alp$ for the orthogonal
reflection on $\R^{2n}$ with root $\alp$.  Given $\beta\in\Psi$ we write
$r_\beta$ for the orthogonal reflection on $\R^{2n}_\sigma$ with root
$\beta$.  We may identify $R_i$ and $r_j$ with $R_{\alp_i}$ and
$r_{\beta_j}$, respectively.
\end{defn}

Recall that $\phi(r_0) = R_n$
and $\phi(r_j) = R_{n-j}R_{n+j}$ for $j\in\{1,\ldots,n-1\}$.

\begin{lemma}\label{lm:sigmaandphi}
The map $\sigma:\R^{2n}\to\R^{2n}$ and the restriction of $\phi$ to
$W(\ddC_n)$ satisfy the following properties for each $w\in W(\ddC_n)$.
\begin{enumerate}[(i)]
\item
$\sigma \phi(w) = \phi(w)\sigma$.
\item
$\phi(w)x  = wx$ if $x\in\R^{2n}_\sigma$.
\end{enumerate}
\end{lemma}

\begin{proof}
As $\sigma R_i \sigma^{-1} = R_{2n-i}$ for each $i\in\{1,\ldots,2n-1\}$, we know that $\sigma R_n \sigma^{-1} = R_{n}$, and $\sigma R_{n+j}R_{n-j}
\sigma^{-1} = R_{n-j} R_{n+j}$, so $\sigma$ commutes with $\phi(r_j)$ for
each $j\in\{0,\ldots,n-1\}$. As $r_0,\ldots,r_{n-1}$ generate $W(\ddC_n)$,
this implies that $\sigma$ commutes with each $\phi(w)$ for $w\in
W(\ddC_n)$.  Hence (i).

Also for (ii) it suffices to verify the statement for $w$ a simple reflection.
Let $x\in\R^{2n}_\sigma$. For $j=0$, we have $\phi(r_0)x =R_nx = r_0 x$ and
for $j=1,\ldots,n-1$, as $(x,\alp_{n-j}) = (x,\alp_{n+j})$,
\begin{eqnarray*}
\phi(r_j)x &=& R_{n-j}R_{n+j}x =
 R_{n-j}(x-(x,\alp_{n+j})\alp_{n+j}) \\
&=& x-(x,\alp_{n+j})\alp_{n+j} -( x,\alp_{n-j})\alp_{n-j}\\
&=& x-  (x,\alp_{n+j}+\alp_{n-j}) (\alp_{n+j}+\alp_{n-j})/2\\
&=&r_jx,
\end{eqnarray*}
as required.
\end{proof}

The following result is well known and gives the connection between the
restriction of $\phi$ to $W(\ddC_n)$ and $\fp$.

\begin{lemma}\label{lm:pfandphi}
The maps $\fp$ and $\phi$ are compatible in the following two ways.
\begin{enumerate}[(i)]
\item
$\fp(\phi(w)\alp) = w \fp(\alp)$  for each $w\in W(\ddC_n)$
and $\alp\in\Phi$.
\item
$\phi(r_{\beta}) = \prod_{\alp\in\fp^{-1}(\beta)}
R_{\alp}$
and
$\phi(e_{\beta}) = \prod_{\alp\in\fp^{-1}(\beta)}
E_{\alp}$
for each $\beta\in\Psi$.
Here, the set $\fp^{-1}(\beta)$ has cardinality 1 or 2, according
to $\beta\in W(\ddC_n)\beta_0$ or
$\beta\in W(\ddC_n)\beta_1$.
\end{enumerate}
\end{lemma}

\begin{proof}
By the definition of $\fp$ and Lemma \ref{lm:sigmaandphi},
\begin{eqnarray*}
\fp(\phi(w)\alp) &=&
(\phi(w)\alp + \sigma(\phi(w)\alp))/2
= (\phi(w)\alp + \phi(w)\sigma(\alp))/2\\
&=& \phi(w) (\alp + \sigma(\alp))/2
= \phi(w) (\fp(\alp))\\
 &=& w (\fp(\alp)),
\end{eqnarray*}
which proves (i).

As for (ii), let $\beta\in\Psi$. The proofs for $r_\beta$ and $e_\beta$ are
almost identical, so we only give the former.
If $\beta$ is simple, the equality
$\phi(r_{\beta}) = \prod_{\alp\in\fp^{-1}(\beta)}
R_{\alp}$
holds by definition of $\phi$.
Suppose $\beta = w\beta_0$ for some $w\in W(\ddC_n)$. Then
$\beta = \phi(w)\beta_0$ by Lemma \ref{lm:sigmaandphi}(ii) and so, by (i)
of the same lemma,
\begin{eqnarray*}
\fp(\beta) &=& \fp(\phi(w)\beta_0) = w\fp(\beta_0) =
 w\beta_0 = \beta.
\end{eqnarray*}
Now
$\fp^{-1}(\beta) = \{\beta\}$ and
\begin{eqnarray*}
\phi(r_\beta) &=& \phi(wr_0w^{-1}) = \phi(w) \phi(r_0) \phi(w)^{-1}
 = \phi(w) R_n \phi(w)^{-1} \\
&=&  R_{\phi(w)\alp_n} =
R_{w\alp_n} = R_\beta = \prod_{\alp\in\fp^{-1}\beta}R_\alp.
\\
\end{eqnarray*}
As $\Psi$ is the union of the two $W(\ddC_n)$-orbits with representatives
$\beta_0$ and $\beta_1$, it only remains to consider $\beta = w\beta_1$ with
$w\in W(\ddC_n)$.
Take $\alp\in\fp^{-1}(\beta)$.
Then
$\fp( \phi(w)\alp_{n-1})  = w \fp( \alp_{n-1}) =
w\beta_1 = \beta = \fp(\alp)$, so, in view of Lemma \ref{lm:sigmaandphi}(i),
$\fp^{-1}(\beta) = \{\phi(w)\alp_{n-1},\phi(w)\alp_{n+1}\}$.
We find
\begin{eqnarray*}
\phi(r_\beta) &=& \phi(wr_1w^{-1}) =
\phi(w) \phi(r_1) \phi(w)^{-1}
=\phi(w) R_{n-1}R_{n+1} \phi(w)^{-1}\\
&=&
R_{\phi(w)\alp_{n-1}}R_{\phi(w)\alp_{n+1}}
= \prod_{\alp\in\fp^{-1}(\beta)}R_{\alp},
\end{eqnarray*}
which establishes (ii).
\end{proof}

We next consider particular sets of mutually orthogonal positive roots in
$\Psi$, and relate them to symmetric admissible sets in $\cA$.

\begin{defn}\label{df:admissible}
Denote by $\cB'$ the collection of all sets of mutually orthogonal roots in
$\Psi^+$ and by $\cA_\sigma$ the subset of $\sigma$-invariant elements of
$\cA$.  As $\fp$ sends positive roots of $\Phi$ to positive roots of $\Psi$,
it induces a map $\fp : \cA_\sigma\to\cB'$ given by $\fp(B) =
\left\{\fp(\alpha)\mid \alpha\in B\right\}$ for $B\in \cA_\sigma$.  An
element of $\cB'$ will be called {\em admissible} if it lies in the image of
$\fp$. The set of all admissible elements of $\cB'$ will be denoted $\cB$.
\end{defn}

\begin{rem}\label{rem:notalladm}
Not all sets of mutually orthogonal roots in $\Psi^+$ are admissible. For
instance $Y = \{\beta_{1}, \beta_{1}+\beta_0\}$ (two mutually orthogonal
short roots) belongs to $\cB'$ (for $n=2$) but not to $\cB$. For, if $X\in
\cA_\sigma$ would be such that $\fp(X) = Y$, then $X$ should contain
$\alpha_1$ and $\alpha_3$ as well as $\alpha_1+\alpha_2$ and
$\alpha_2+\alpha_3$; but these roots are not mutually orthogonal.  On the
other hand, for $n\ge4$, the unordered pair $\{\beta_1,\beta_3\}$ from
another $W(\ddC_n)$-orbit of mutually orthogonal short roots, is the image
of the admissible set
$\{\alpha_{n-1},\alpha_{n+1},\alpha_{n-3},\alpha_{n+3}\}$ and so belongs to
$\cB$.

Also
$\{\beta_{0}, 2\beta_{1}+\beta_0\}$ (two mutually orthogonal
long roots) does belong to $\cB$ (for $n=2$) as it coincides with
$\fp(X)$, where $X = \{\a_n,\a_{n-1}+\a_n+\a_{n+1}\}$.
\end{rem}

\begin{prop}\label{3.5}
The monoid $\BrM(\ddC_n)$
acts on $\cA_\sigma$ under
the composition of the above action and $\phi$.
\end{prop}

\begin{proof} It suffices
to prove that $\cA_\sigma$ is closed under the action of $\phi(\BrM(\ddC_n))$.
It is easy to see that $\sigma(a)\sigma(B)=\sigma(aB)$, for $a\in
\BrM(\ddA_{2n-1})$ and $B\in \cA$. Consequently,
if $a\in \BrM(\ddA_{2n-1})_\sigma$ and $B\in \cA_\sigma$, then it follows that 
$aB=\sigma(a)\sigma(B)=\sigma(aB)$. This shows $aB\in \cA_\sigma$.  As
$\phi(\BrM(\ddC_n))\subseteq \BrM(\ddA_{2n-1})_\sigma$, the proposition
follows.
\end{proof}

\begin{prop}\label{lm:fp_injective}
The map $\fp: \cA_\sigma\to\cB$ is bijective
and $W(\ddC_n)$-equivariant, so $\fp(\phi(w)X) =
w\fp(X)$ for $X\in\cA_\sigma$ and $w\in W(\ddC_n)$.
\end{prop}

\begin{proof}
The map $\fp$ is surjective by definition of $\cB$.
Let $Y\in\cB$ and $X\in\fp^{-1}(Y)$.
If $\beta \in Y$, then there is $\alp\in X$ such that
$\beta=\fp(\alp)$. As $X\in\cA_\sigma$, it follows that
$\sigma\alp\in X$, so $X = \{\alp\in\Phi\mid \fp(\alp) \in Y\}$
is uniquely determined by $Y$. This shows that $\fp$ is injective.

Finally, if in addition, $w\in W(\ddC_n)$, then
$\phi(w)X\in\cA_\sigma$ by Proposition \ref{3.5}, and,
by Lemma \ref{lm:pfandphi},
\begin{eqnarray*}
\fp(\phi(w)X)
&=& \{\fp(\phi(w)\alp)\mid\alp\in X\}
= w \{\fp(\alp)\mid\alp\in X\}
= w \fp(X).
\end{eqnarray*}
\end{proof}

\begin{lemma}\label{lm:SimpleRootRels}
Let $i$ and $j$ be nodes of the Dynkin diagram $\ddC_{n}$.
If $w\in W(\ddC_n)$ satisfies $w\beta_i=\beta_j$, then
$we_iw^{-1}=e_j$.
\end{lemma}

\begin{proof}
Observe that $w\b_i = \b_j$ only holds for
distinct $i$ and $j$ if $i,j>0$, in which case the existence of such a $w$
is a direct consequence of known results on Coxeter groups.
The full statement then follows from the fact that all generators
of the stabilizer of
$\b_i$ in $W(\ddC_n)$ also stabilize $e_i$, which we prove now.

Suppose  $i>0$.
As
\begin{eqnarray*}r_1e_1r_1&\overset{(\ref{0.1.4})}{=}&e_1,\\
r_0(r_1r_0e_1r_0r_1)r_0&\overset{(\ref{0.1.14})+(\ref{0.1.17})}{=}&r_0r_0e_1 r_0r_0=e_1,\\
y_3 e_1 y_3&\overset{\ref{X_n}}{=}&e_1y_3 y_3=e_1,\\
r_ie_1r_i&\overset{(\ref{0.1.9})}{=}&e_1r_i r_i=e_1,\, \mbox{for}\,\, i>3,
\end{eqnarray*}
the elements $r_1$, $r_0r_1r_0$, $y_{3}$, $r_3$,$\ldots$, $r_{n-1}$
stabilize $e_1$. But these elements are known to generate the full
stabilizer
in $W(\ddC_n)$ of $\beta_1$,
so $w\beta_1=\beta_1$ implies
$we_1w^{-1}=e_1$ for every $w\in W(\ddC_n)$.

If one of $i$ and $j$ is $0$, then they both are, as $(we_iw^{-1})^2 =
\delta^k we_iw^{-1}$, where $k = 2$ if $i>0$ and $k=1$ otherwise
(see (\ref{0.1.5}) and (\ref{0.1.6})).

It is known that the stabilizer of $\beta_0$ is generated by $r_0$,
$r_1r_0r_1$, and $r_i$ $(i=2,\dots, n-1)$.  These elements also centralize
$e_0$:
\begin{eqnarray*} r_0e_0r_0&\overset{(\ref{0.1.4})}{=}&e_0,\\
 r_1r_0r_1e_0r_1r_0r_1&\overset{(\ref{0.1.20})}{=}&e_0,\\
 r_ie_0r_i&\overset{(\ref{0.1.8})}{=}&e_0 \qquad \mbox{for}\,\, i>1.
\end{eqnarray*}
This ends the proof of the lemma.
\end{proof}

Consider a positive root $\beta$ and a node $i$ of type $\ddC_n$. If there
exists $w\in W$ such that $w\beta_i=\beta$, then we can define the element
$e_{\beta}$ in $\BrM(\ddC_n)$ by
$$e_{\beta}=we_iw^{-1}.$$
The above lemma implies that $e_\beta$ is well defined. In general,
$$we_{\beta}w^{-1}=e_{w\beta},$$ for $w\in W(\ddC_n)$ and $\beta$ a root
of $W(\ddC_n)$. Note that $e_{\beta}=e_{-\beta}$ in view of (\ref{0.1.4}).

In this perspective, we can reinterpret the element $y_i$ of (\ref{df:yn})
as $r_\gamma$ and $z_i$ of (\ref{df:zn}) as $e_\gamma$, where $\gamma =
\beta_0+2\beta_1+\cdots +2\beta_{i-1}$.
Proposition \ref{X_n} shows that, for each $i\in\{0,\ldots,n-1\}$,
\begin{eqnarray}\label{eq:bi}
b_i &=&\prod_{k=1}^{i} z_k 
\end{eqnarray}
is well defined. 
The  admissible set
\begin{eqnarray}\label{eq:Bi}
B_i &=&\fp^{-1} \left(
\{\b_0,\beta_0+2\beta_1,\ldots,\beta_0+2\beta_1+\cdots
+2\beta_{i-1}
\}\right)
\end{eqnarray}
in $\cA_\sigma$ is both the top and the bottom of $\phi(b_i)$.
In other
words, the symmetric diagram of $b_i$ has horizontal strands from
$(n+1-j,1)$ to $(n+j,1)$ and from $(n+1-j,0)$ to $(n+j,0)$ for each
$j\in\{1,\ldots,i\}$.  This is the special case $p=i$ of the top
displayed in Figure \ref{fig:Bip}.
Later, below (\ref{eq:bpip}), we will use this fact. 
The $B_i$ $(i=0,\ldots,n)$ are a complete set of $W(\ddA_m)$-orbit
representatives in~$\cA$. Moreover, $\phi(b_i)$ has height
$0$.

\begin{prop} \label{lem:rel}
Let $\beta$ and $\gamma$ be positive roots of $\Psi$.
\begin{enumerate}[(i)]
\item 
$e_\beta r_\beta =r_\beta e_\beta=e_\beta$,
$e_\beta^2=\delta^2e_\beta$  if $\beta$ is short,
and $e_\beta^2=\delta e_\beta$  if $\beta$ is long.

\item 
If $(\beta,\gamma)=\pm 1$  and  $\beta$ and $\gamma$ are short,
then
$e_\beta r_\gamma e_\beta=e_\beta$,
$r_\beta r_\gamma e_\beta=e_\gamma r_\beta r_\gamma=e_\gamma e_\beta$,
and $e_\beta e_\gamma e_\beta=e_\beta$.

\item
If $(\beta,\gamma)=\pm 1$  with $\beta$ short and $\gamma$ long, then the
equations (\ref{0.1.11})--(\ref{0.1.18}) and (\ref{4.1.2})--(\ref{4.1.5})
still hold with the subscripts $1$ and $0$
replaced by $\beta$ and $\gamma$, respectively.
\item 
If $(\beta,\gamma)=0$ and $\beta$ and $\gamma$ are both long,
then $e_\beta e_\gamma =e_\beta e_\gamma$.
\item 
If $(\beta,\gamma)=0$  and $\beta$ and $\gamma$ are both short,
and there exists a long positive  root $\a$
such that $\beta=\gamma+\a$ or $\beta=\gamma-\a$, then
$e_\beta e_\gamma=\delta r_{\a} e_{\gamma}$ or
       $ e_\gamma e_\beta=\delta r_{\a} e_\beta$.
In each case, $e_\beta e_\gamma\ne e_\gamma e_\beta $.
\end{enumerate}
\end{prop}

\begin{proof}  The assertions
are easily proved after reduction to simple cases using Lemma
\ref{lm:SimpleRootRels}.
We illustrate the argument by treating (v) now in greater detail.

Up to an interchange of $\beta$ and $\gamma$, there exists $w\in
W(\ddC_n)$ such that $w\beta_0=\alpha$ and $w\beta_1=\gamma$ and $\beta =
\gamma+\a=r_{\a}\gamma$.  Now $e_\beta e_\gamma = r_\a e_\gamma r_\a
e_\gamma =w r_0e_1r_0e_1 w^{-1} \overset{(\ref{0.1.12})}{=} \delta w r_0e_1
w^{-1} = \delta r_\a e_\gamma$.  The inequality stated at the end of part (v) follows
from the inspection of symmetric Brauer diagrams in the image of $\phi$.
\end{proof}

As a consequence of Proposition \ref{lem:rel}(iv) and (\ref{0.1.9}),
the product of $e_\beta$, for $\beta$ running over the members of an
admissible set, does not depend on the order. Therefore,
for each $B\in\mathcal{B}$,
we may define

\begin{eqnarray}
\label{df:eB}
 e_{B}&=&\prod_{\beta\in B} e_{\beta},
\end{eqnarray}

This is very similar to the definition of $E_B$ in (\ref{eqn:EhatB}).  Part
(v) of the proposition shows that it is essential that $B$ be admissible for
a set $B$ of orthogonal roots to define a product as in (\ref{df:eB}).
There exists another $W(\ddC_n)$-orbit of pairs of mutually orthogonal positive
roots than those in part (v) of the previous proposition; these lead to admissible sets and behave well
in (\ref{df:eB}) in view of (\ref{0.1.9}).

As $W(\ddC_n)$ is a subgroup of the monoid $\BrM(\ddC_n)$, it also acts on
$\cA_n$ (cf.~Proposition \ref{3.5}). We will show that the admissible sets
defined below are orbit representatives for this action.

\begin{defn}\label{rem:B_{i,p}}
For $i$ and $p$ with $0\le p\le i\le n$ and $i-p$ even, write
\begin{eqnarray}\label{fip}
e_{i,p} &=& e_{p+1}e_{p+3}\cdots e_{i-1},
\end{eqnarray}
and
\begin{eqnarray}
B_{i,p}&=&
e_{p+1}e_{p+3}\cdots e_{i-1}B_i.
\nonumber
\end{eqnarray}
In addition, for $p'\in\N$ with $0\le p'\le i<n$ and 
$i-p'$ even, we write
\begin{eqnarray}
\label{eq:bpip}
b_{p,i,p'} & = & e_{i,p}b_ie_{i,p'},
\end{eqnarray}
where $b_i$ is as defined in (\ref{eq:bi}).
\end{defn}

\begin{figure}[h!]
\begin{center}
\includegraphics[width=.9\textwidth]{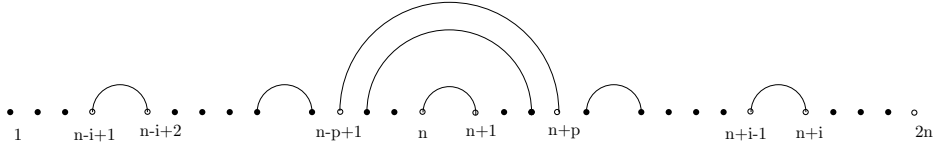}
\end{center}
\caption{The admissible set $B_{i,p}$.}
\label{fig:Bip}
\end{figure}

Observe that $B_{i,i}=B_i$.  The admissible set $B_{i,p}$ is pictured in
Figure \ref{fig:Bip} as the top of a Brauer diagram.

\begin{lemma}
\label{lm:BipAct}
If $0<p,p'<i$ with $i-p$ and $i-p'$ even, then
$\phi(b_{p,i,p'})$ is a diagram of height $0$
with top $B_{i,p}$ and bottom $B_{i,p'}$.
Moreover,
$\phi(b_{p,i,p'}) = a_{B_{i,p},B_{i,p'}}$.
\end{lemma}

\begin{proof} This is an easy verification involving the
left and right monoid actions of Proposition \ref{3.5}
(use that $e_{i,p'}^{\op} = e_{i,p'}$).
\end{proof}

Recall that the Temperley--Lieb subalgebra of $\Br(\ddC_n)$ is the subalgebra generated by $e_0,\ldots,e_{n-1}$, and similarly for $\Br(\ddA_m)$ and for
the corresponding monoids.  The following lemma implies that the
restriction of $\phi$ to the Temperley--Lieb subalgebra of $\Br(\ddC_n)$
behaves as required.

\begin{lemma}\label{lm:sigmaX}
Let $B\in\cA_\sigma$. Then $E_B=\phi(e_{\fp(B)})$ and, in particular, hence $E_B\in\phi(\Br(\ddC_n))$.
Moreover, the restriction of $\phi$ to the Temperley--Lieb subalgebra of
$\Br(\ddC_n)$ is surjective onto the intersection of the Temperley--Lieb
algebra of $\Br(\ddA_{2n-1})$ with $\SBr(\ddA_{2n-1})$.
\end{lemma}

\begin{proof}
Suppose $B\in\cA_\sigma$.
Then, by Lemmas \ref{lm:pfandphi}(ii) and \ref{lm:phiIsHomo},
\begin{eqnarray*}
E_B&=&\prod_{\alp\in B} E_\alp
=\prod_{\beta\in \fp(B)}  E_{\fp^{-1}(\beta)}
=\prod_{\beta\in \fp(B)} \phi (e_\beta)
=\phi\left(\prod_{\beta\in \fp(B)} e_\beta\right)
=\phi(e_{\fp(B)}),
\end{eqnarray*}
as required for the first statements.

Let $\bar i\in \{0,1\}$ be such that $i\equiv \bar i \pmod 2$.  Then
$B_{i,\bar i}$ consists of simple roots only and belongs to the same
$W(\ddA_{2n-1})$-orbit in $\cA$ as $B_i$ or in fact $B_{i,p}$ for
any $p$.

As for the last statement, suppose that $a\in \SBr(\ddA_{2n-1})$ is a
monomial in the Temperley--Lieb subalgebra of $\Br(\ddA_{2n-1})$. Then, up
to powers of $\delta$, there are $i\in\{0,\ldots,n\}$ and
$B,B'\in\cA_\sigma$ of height $0$ with $a = a_{B,B'}= (a_{B,B_{i,\bar i}}
)(a_{B',B_{i,\bar i}})^\op$ in the notation of and by use of Lemma
\ref{lm:hetB}.  In view of the opposition involution and the first statement
of this proposition, it suffices to show that $a := a_{B,B_{i,\bar i}}$ 
lies in the
image under $\phi$ of the submonoid of $\BrM(\ddC_n)$ generated by
$e_0,e_1,\ldots,e_{n-1}$, the Temperley--Lieb submonoid of
$\BrM(\ddC_n)$. To this end, let $B\in\cA_\sigma$ be of height $0$. 
We will establish the existence of an element $b$ in this Temperley--Lieb
submonoid with $\phi(b)B_{i,\bar i} = B$. This will suffice as then
Theorem \ref{th:AmDecomp} gives
$a= E_B\phi(b)=\phi(e_{\fp(B)}b)$, up to powers
of $\delta$, and so $a\in\phi(\BrM(\ddC_n))$.

For $\gamma_1=\epsilon_{i_1}-\epsilon_{j_1}$ and
$\gamma_2=\epsilon_{i_2}-\epsilon_{j_2}\in \Phi^+$,  we say $\gamma_1\ll
\gamma_2$ if $1\le i_2<i_1<j_1<j_2\le 2n$.  If $\gamma_1\ll
\gamma_2\ll\cdots\ll\gamma_s$, and $\gamma_k\in B$, for $1\le k\le s$, we
say that $\gamma_1\ll \gamma_2\ll\cdots\ll\gamma_s$ is a \emph{chain in}
$B$, and call $s$ the \emph{length} of this chain.  Now we use induction on
the maximal length $t$ of a chain in $B$ and the number $l$
of longest chains in $B$.

Suppose first $t=1$. Then $B$
consists of simple roots and
$b$ as required can be found
easily.
Suppose $t>1$.  Let $\gamma_1\ll
\gamma_2\ll\cdots\ll\gamma_t$ be a longest chain of $B$ with
$r_2=\epsilon_{j_2}-\epsilon_{i_2}$. Since $a$ is an element of
the Temperley--Lieb submonoid of $\BrM(\ddA_{2n-1})$,
$j_2-i_2$ is odd, and for $1\le k\le r=[(j_2-i_2)/2]$ we have
$\epsilon_{i_2+2k-1}-\epsilon_{i_2+2k}\in B$.  Let
\begin{eqnarray*}
D&=&\{\gamma_2\}\cup\{\epsilon_{i_2+2k-1}-\epsilon_{i_2+2k}\}_{k=1}^r,\\
D'&=&\{\epsilon_{i_2+2k}-\epsilon_{i_2+2k+1}\}_{k=0}^r,\\
B'&=&(B\setminus D\cup\sigma(D))\cup (D'\cup\sigma(D')) .
\end{eqnarray*}
Now $B'\in \cA_\sigma$ is the top of a Temperley--Lieb element with maximal
length less than $t$ or number of maximal chains fewer than $l$. By the
induction hypothesis there exists some Temperley--Lieb element $b'$
in $\BrM(\ddC_n)$ with $\phi(b')B_i=B'$. It satisfies
$$E_X(D'\cup\sigma(D'))=D\cup\sigma(D)\qquad \mbox{ and } \qquad
E_X(B')=B,$$ where $X=(D\setminus\{\gamma_2\})\cup
\sigma(D\setminus\{\gamma_2\})\in \cA_\sigma$. As $E_X =\phi(e_{\fp(X)})$
by the first statement of the lemma, we conclude that
$\phi(e_{\fp(X)}b')B_{i,\bar i} = B$, which proves the lemma.
\end{proof}

\begin{defn}\label{df:Ki}
  Let $i\in\{0,\ldots,n\}$. Define $\hE^{(i)}$ to be the idempotent in
  $\BrM(\ddA_{2n-1})$ corresponding to $b_i$;
\begin{eqnarray*}
\hE^{(i)} &:=& \delta^{-i} \phi(b_i).
\end{eqnarray*}
In addition, define
\begin{eqnarray}\label{eq:ehat}
\hb^{(i)} &:=& \delta^{-i}e_{\fp(B_i)}.
\end{eqnarray}
Then $\hE^{(i)} = \phi(\hb^{(i)})$ by Lemma \ref{lm:sigmaX}.
Furthermore, we write $K_i$ instead of $K_{B_i}$ as introduced
in Definition \ref{df:KB}. 
\end{defn}

Observe that $K_i$ is generated by
\begin{eqnarray}\label{eq:KiGens}
&&\hskip-0.7cm
\hE^{(i)}R_{n\pm\lfloor 2+ i/2\rfloor},
\hE^{(i)}R_{n\pm\lfloor 3+ i/2\rfloor},
\ldots,
\hE^{(i)}R_{n\pm (n-1)},\mbox{\  and \ }
\hE^{(i)}R_{0},
\end{eqnarray}
where $R_0$ stands for the longest reflection of $W(\ddA_{2n-1})$, that is,
$$(R_1R_{2n-1}) (R_2R_{2n-2})\cdots (R_{n-1}R_{n+1}) R_n
(R_{n-1}R_{n+1}) \cdots  (R_2R_{2n-2})(R_1R_{2n-1}).
$$
The definition of $y_i$ from (\ref{df:yn})
shows that $\phi(y_n) = R_0$.

We will now prove that the $\sigma$-fixed part of $K_i$ is contained in the
image of $\phi$, making use of the fact that $K_i$ is a Coxeter group of which the above
generators are a Coxeter system, as given by Lemma \ref{lm:heightZero}.

\begin{lemma}\label{lm:Ki}
Let $i\in\{0,\ldots,n\}$. The set (\ref{eq:KiGens}) of simple reflections of
$K_i$ is invariant under $\sigma$.  In fact, $\sigma$ induces the nontrivial
automorphism on the Coxeter type $\ddA_{2(n-j)-1}$ of $K_i$,
where $j = 1+ \lfloor i/2\rfloor $.  As a consequence, the subgroup of
$\sigma$-fixed elements of $K_i$ is of type $\ddC_{n-j}$ and is generated by
the images under $\phi$ of
$ \hb^{(i)}r_{j+1}, \hb^{(i)}r_{j+2}, \ldots,
\hb^{(i)}r_{n-1}$,
and
$\hb^{(i)}y_{n}\hb^{(i)}$.
\end{lemma}

\begin{proof}
Clearly, $\sigma$ fixes $\hE^{(i)} = \phi(\hb^{(i)})$.
Moreover, it fixes $R_0$ and interchanges $R_{n- k}$ and
$R_{n+k}$, so indeed the Coxeter system of $K_i$ is $\sigma$-invariant and
$\sigma$ induces the nontrivial automorphism on the Coxeter type
$\ddA_{2(n-j)-1}$ of $K_i$. It is well known (cf.~\cite{Muehl92})
that the subgroup of $\sigma$-fixed elements of $K_i$ is generated by
the Coxeter system $$\hE^{(i)}R_{j+1}R_{2n-j-1},
\hE^{(i)}R_{j+2}R_{2n-j-2},
\ldots,
\hE^{(i)}R_{1}R_{2n-1},\mbox{ \  and \ }
\hE^{(i)}R_{0}
$$
of type $\ddC_{n-j}$.
These generators coincide
with the $\phi$-images of the simple reflections
in the statement of the lemma.
\end{proof}

The case $i=0$ of the above lemma confirms that the restriction of
$\phi$ to $W(\ddC_n)$ is an embedding of this group into
$W(\ddA_{2n-1})$ whose image coincides with the
$\sigma$-fixed elements of $W(\ddA_{2n-1})$.

The $W(\ddA_{2n-1})$-orbit of $B_{i,p}$ contains $B_i$, but, for $p<i$,
these two admissible sets are in distinct $W(\ddC_n)$-orbits: For $B\in\cB$,
the numbers $i$, the size of $B$, and $p$, the number of roots in $B$ fixed
by $\sigma$, are constant on the $W(\ddC_n)$-orbit of $B$ in $\cA_\sigma$.
They actually determine this orbit uniquely.

\begin{prop}\label{prop:tran}
Let $B\in\cA_\sigma$ be such that the number of $\sigma$-fixed roots in
$B$ is equal to $p$ and such that $B$ has cardinality $i$.  Then there exists
an element $w$ of the subgroup $W(\ddC_n)$ of $W(\ddA_{2n-1})$ such that
$wB_{i,p}=B$.
\end{prop}

\begin{proof}
By Theorem \ref{th:AmDecomp} and the case $i=0$ of Lemma \ref{lm:Ki}, it
suffices to find a symmetric diagram $w$ in $\BrM(\ddA_{2n-1})$ without
horizontal strands, moving $B_{i,p}$ to $B$.

For each $\gamma\in B$ with $\gamma\ne \sigma(\gamma)$, where
$\gamma=\alpha_{t}+\alpha_{t+1}+\cdots+\alpha_s$, $1\leq t\leq s\leq 2n-1$,
we draw four vertical strands in $w$ as follows: from $(t,1)$ to $(k,0)$,
from $(2n+1-t,1)$ to $(2n+1-k,0)$, from $(s+1,1)$ to $(k+1,0)$, and from
$(2n-s,1)$ to $(2n-k,0)$ with $\{\alpha_k,\alpha_{2n-k}\}\subset
B_{i,p}$. For each $\gamma\in B$ with $\gamma=\sigma(\gamma)$, where
$\gamma=\alpha_{n-t}+\alpha_{n-t+1}+\cdots+\alpha_{n+t}$, $0\leq t\leq n-1$,
we draw two vertical strands: from $(n-t,1)$ to $(n-k,0)$, and from
$(n+t+1,1)$ to $(n+k+1,0)$ where
$\alpha_{n-k}+\alpha_{n-k+1}+\cdots+\alpha_{n+k}\in B_{i,p}$.  Between the
remaining $2n-2i$ dots at the top and $2n-2i$ nodes at the bottom, we just
draw vertical strands in such a way that these strands do not cross. This
provides the required diagram $w$.
 \end{proof}

\begin{prop}\label{prop:phisurj}
The homomorphism $\phi: \Br(\ddC_n)\to \SBr(\ddA_{2n-1})$ is surjective.
\end{prop}

\begin{proof}
It suffices to prove the statement for
the corresponding monoids. So, let $a$ be an element of
$\BrM(\ddA_{2n-1})$ with $\sigma(a) = a$.
Then, by Theorem \ref{th:AmDecomp},
up to replacing $a$ by a power of $\delta$, we have
\begin{eqnarray}
\label{eq:a}
a &=& U\hE^{(i)}VW
\end{eqnarray}
for some
$i\in\{ 0,\ldots, n\}$,
$U\in \BrM(\ddA_{2n-1})\hE^{(i)}$,
$W\in \hE^{(i)}\BrM(\ddA_{2n-1})$, and $V\in K_i$
such that $\het(U)+\het(V)+\het(W) = \het(a)$.
Here $B_i$
and $K_i$ are as in Definition \ref{df:Ki}.
As noted before, $\sigma(B_{i}) = B_{i}$,
which implies $\sigma(\hE^{(i)}) = \hE^{(i)}$.
{From} Lemma \ref{lm:sigmaX}, it follows that
$\hE^{(i)}\in\phi(\BrM(\ddC_n))$. Now $\phi$ is a homomorphism, so it suffices
to show that $U$, $V$, $W$ are in the image of $\phi$.

As $\sigma(a) = a$ we find $U\hE^{(i)}VW = \sigma(U) \hE^{(i)}
\sigma(V)\sigma(W)$ with $\sigma(U)\in \BrM(\ddA_{2n-1})\hE^{(i)}$,
$\sigma(W)\in \hE^{(i)}\BrM(\ddA_{2n-1})$ and $\sigma(V)\in K_i$ (note that
$\sigma(K_{i}) = K_i$ by Lemma \ref{lm:Ki}) such that
$\het(\sigma(U))+\het(\sigma(V))+\het(\sigma(W)) = \het(a)$.  According to
Theorem \ref{th:AmDecomp}, the expression in (\ref{eq:a}) is unique, which
implies $\sigma(U)= U$, $\sigma(V)=V$, and $\sigma(W) = W$.  {From} Lemma
\ref{lm:Ki}, we find $V\in\phi(\BrM(\ddC_n))$.  Writing $U\hE^{(i)}VW =
(U\hE^{(i)})V (W\hE^{(i)})^{\op}$ and using Proposition \ref{prop:opp} (observe
that the anti-involution $x\mapsto x^\op$ commutes with $\phi$ and
$\sigma$), we may restrict ourselves to the case where $a =
U$ with $\het(U) = \het(a\emptyset)$. Therefore, we will assume that
$a$ is of this kind.

Put $B = a\emptyset$. Then $B\in\cA_\sigma$ and so there are
$w\in W(\ddA_{2n-1})_\sigma = \phi(W(\ddC_n))$ and
$p,i\in\{0,1,\ldots,n\}$ with $0\le p\le i$ and $i-p$ even such that
$B = w B_{i,p}$.

If $B = B_{i,p}$, then, according to Lemma \ref{lm:BipAct} and Theorem
\ref{th:AmDecomp},
$U = \phi(e_{p+1}e_{p+3}\cdots e_{i-1})\hE^{(i)}$, which
belongs to $\phi(\BrM(\ddC_n))$.
Denote this element by  $E^{(i,p)}$.

Now, in the general case,
$wE^{(i,p)} = U\hE^{(i)}V$ for some $V\in K_i$ such that
$\het(V) = \het(wE^{(i,p)}) - \het(B)$. By Theorem \ref{th:AmDecomp}
the element
$V$ of $K_i$ is uniquely determined by $w$, so
$ U\hE^{(i)}V = wE^{(i,p)} = \sigma(w)\sigma(E^{(i,p)}) =
\sigma(wE^{(i,p)}) = U\hE^{(i)}\sigma(V)$ implies $V = \sigma(V)$.
The inverse $V'$ of $V$ in the group $K_i$ with unit $\hE^{(i)}$ satisfies
$wE^{(i,p)}V' = U\hE^{(i)} = U =  a$.
As $V'$ is again uniquely determined by $V$, we have $\sigma(V') = V'$
and so Lemma \ref{lm:Ki} gives $V'\in\phi(\BrM(\ddC_n))$.
We conclude
$a = wE^{(i,p)}V' \in \phi(\BrM(\ddC_n))$.
\end{proof}

\begin{section}{Admissible sets and their orbits} \label{sect:adm}
We continue with the study of the Brauer monoid $\BrM(\ddC_n)$ of type
$\ddC_n$ acting on $\cA_\sigma$, the subset of $\cA$ of $\sigma$-invariant
admissible sets.  This leads to a normal form for elements of $\BrM(\ddC_n)$
to the extent that we can provide an upper bound on the rank of
$\Br(\ddC_n)$. The bound found in Theorem \ref{th:uniqueform} is
instrumental in the proof at the end of this section of the main Theorem \ref{thm:main}.

\begin{lemma} \label{lem:numip}
Let $i\in\{0,\ldots,n\}$ and $p\in\{0,\ldots,i\}$ be such that
$q=(i-p)/2$ is an integer.
Then the $W(\ddC_n)$-orbit of $B_{i,p}$ has size ${n!}/(p!q!(n-i)!)$.
\end{lemma}

\begin{proof}
By Proposition \ref{prop:tran}, the
cardinality of the orbit of $B_{i,p}$ under $W(\ddC_n)$ is
equal to the number of diagram tops with
$i$ horizontal strands of which precisely
$p$ strands are fixed by $\sigma$. This number is readily seen to be
\[{n\choose{p}}{n-p\choose{2q}} (4q-2) (4q-6) \cdots 2
\quad=\,\,\frac{n!}{p!q!(n-i)!}.\]
\end{proof}

\begin{cor}\label{lemma:anformproof}
The rank $a_{2n}$ of $\SBr(\ddA_{2n-1})$ satisfies
\begin{eqnarray*}
a_{2n}&=&\sum_{i =0}^n
    \left(\sum_{p+2q = i}
               \frac{n!}{p! q! (n-i)!}
    \right)^2
     \,2^{n-i}\, (n-i)!.
\end{eqnarray*}
\end{cor}

\begin{proof} If in a symmetric diagram with $2i$ horizontal strands,
all horizontal strans are fixed,
the remaining $2(n-i)$ vertical strands will be in one to one correspondence
with the elements of the Weyl group of type $\ddC_{n-i}$ and order
$2^{n-i}(n-i)!$. Therefore the
corollary follows from Lemma \ref{lem:numip}.
\end{proof}

\noindent Now we proceed to describe the stabilizer in $W(\ddC_n)$ of $B_{i,p}$.

\begin{defn}
\label{df:Ai}
Let $i\in\{0,\ldots,n\}$ and $p\in\{0,\ldots,i\}$
be such that $i-p = 2q$ for some $q\in\N$.
By $A_{i,p}$ we denote the subgroup of $W(\ddC_n)$ generated by the
following elements:
\begin{eqnarray*}
&&r_j\ (j=0,\dots, p-1),
\\
&&r_{p+2k-1}\ (k=1,\dots, q),
\\
&&t_{0,p}=y_{p+1}r_{p+1}y_{p+1}, \\
&&t_{k,p}=r_{p+2k}r_{p+2k-1}r_{p+2k+1}r_{p+2k}\ (k=1,\dots,q-1).
\end{eqnarray*}
Furthermore, by $L_i$ we denote the subgroup of $W(\ddC_n)$
generated by $y_{i+1}$,  $r_{i+1}$, $\ldots$, $r_{n-1}$.
Finally,
we set $N_{i,p}=\langle A_{i,p}, L_i\rangle $
and let
$D_{i,p}$ be a fixed set of representatives for left cosets of
$N_{i,p}$ in $W(\ddC_n)$.
\end{defn}
Figure \ref{fig:t} depicts $B_{6,2}$ as a top and
two elements of the form $t_{k,p}$.

\begin{figure}[h!]
\begin{center}
\includegraphics[width=.9\textwidth]{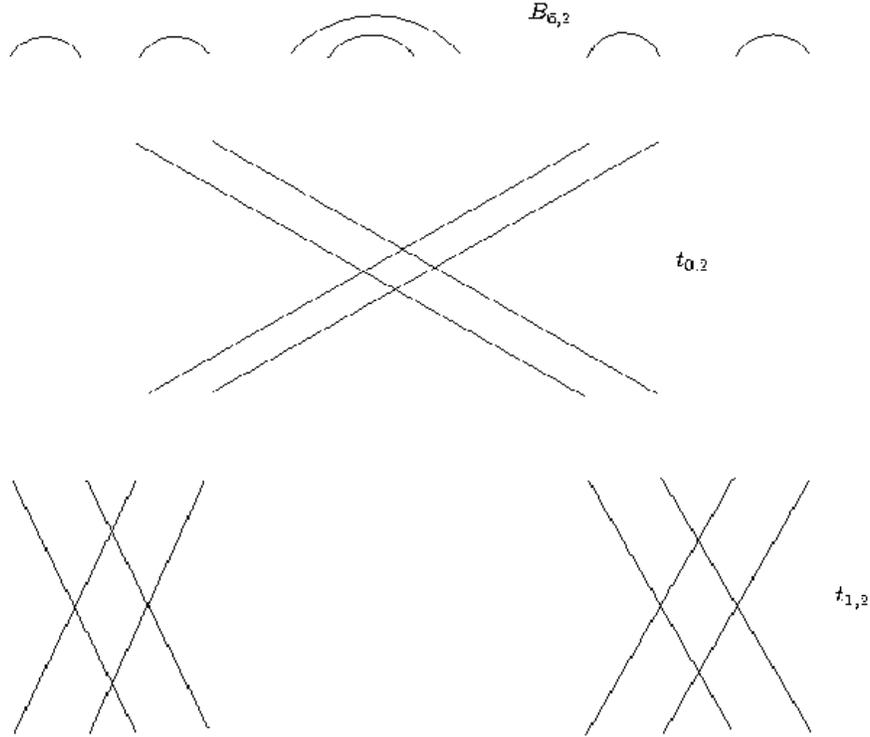}
\end{center}
\caption{The $\phi$-images of the elements  $t_{0,2}$ and $t_{1,2}$.}
\label{fig:t}
\end{figure}

It is easy to check that the generators of $A_{i,p}$ and $L_i$, and hence
the whole group $N_{i,p}$ leaves $B_{i,p}$ invariant.  The next lemma shows that
$N_{i,p}$ is the full stabilizer of $B_{i,p}$ in $W(\ddC_n)$.

\begin{lemma}\label{lem:Li}
The subgroup of $W(\ddC_n)$
generated by $\{t_{j,p}\}_{j=0}^{q-1}$ is isomorphic to $W(\ddC_{q})$ and
the cardinality of $A_{i,p}$ is $2^{i}p!q!$.
Moreover, $L_i$ is isomorphic to $W(\ddC_{n-i})$.
Furthermore, $N_{i,p}$
is the stabilizer of $B_{i,p}$ in $W(\ddC_n)$ and
isomorphic to $ A_{i,p}\times L_i$.
\end{lemma}

\begin{proof} Put
\begin{eqnarray*}
&&A=\langle r_0, r_1,\ldots,r_{p-1} \rangle,\\
&&B=\langle t_{0,p}, t_{1,p},\ldots,t_{q-1,p} \rangle,\\
&&C=\langle r_{p+1}, r_{p+3},\ldots, r_{i-1} \rangle.
\end{eqnarray*}
Being a parabolic subgroup of type $\ddC_p$, the group $A$ is isomorphic to
$W(\ddC_{p})$.  Since the supports of the simple reflections involved in $A$
lie in $\{0,\ldots,p-1\}$ and those of $B\cup C$ lie in $\{p+1,\ldots,i\}$,
each element of $A$ commutes with each element of $B\cup C$.  Now we claim that $B$ is isomorphic to $W(\ddC_{q})$. Ignoring the
$2p$ vertical strands in the middle of generators of $B$ in
$\BrM(\ddA_{2n-1})$, and comparing them with canonical generators of
$W(\ddC_{q})$ in $\BrM(\ddA_{2q-1})$ gives an easy pictorial proof of our
claim.

Consider the diagrams of elements of $B$ and $C$ in $\BrM(\ddA_{2n-1})$.
Each diagram of $B$ only has non-crossing strands starting from $(n-k-1,1)$
and $(n-k+1,1)$, for $k=p+1$, $p+3, \ldots, i-1$, which can never occur in
nontrivial elements of $C$. Hence $B\cap C=\{1\}$.  For the generators of
$B$ and $C$, the following equations hold.
\begin{eqnarray*}
t_{0,p}r_{p+2k-1}t_{0,p}&=&r_{p+2k-1}, \,\mbox{for } \, 1\leq k\leq q
\\
t_{k,p}r_{p+2k-1}t_{k,p}&=& r_{p+2k+1}, \,\mbox{for }\, 1\leq k\leq q
\\
t_{s,p} r_{p+2k-1}t_{s,p}&=& r_{p+2k-1},\,\mbox{for } 1\leq s,\, k \leq q,\, \mbox{and}\, |s-k|>1.
\end{eqnarray*}
Therefore the subgroup $BC$ in $W(\ddC_n)$ is the semiproduct of $C$ and $B$
with $C$ normal.  Consider the diagrams of elements of $BC$ and $A$ in
$\Br(\ddA_{2n-1})$.  Each element in $BC$ keeps the $2p$ strands in the
middle invariant, but each element of $A$ keeps the left $2n-2p+2$ strands
invariant.  Therefore $A\cap BC=\{1\}$.  Thus, $A_{i,p}=BC\times
A$, and hence
$$|A_{i,p}|=|B||C||A|=2^{i}p!q!.$$

The reflections $y_{i+1}$, $r_i$, $r_{i+1}$, $\ldots$, $r_{n-1}$ have roots
$\beta_0+2\beta_1+\cdots+2\beta_i$, $\beta_{i+1}$, $\ldots$, $\beta_{n-1}$,
respectively, which form a simple root system of type
$\ddC_{n-i}$. Therefore the subgroup $L_i$ is isomorphic to $W(\ddC_{n-i})$.

In the Coxeter diagram $\ddA_{2n-1}$ we see that $A_{i,p}\cap L_i=1$ and
all elements in $L_i$ commute with all elements in $A_{i,p}$, so
$N_{i,p}$ is the direct product of $L_i$ and $A_{i,p}$. This gives
\begin{eqnarray*}
|D_{i,p}| &=&\frac{|W(\ddC_n)|}{|A_{i,p}| |L_i|}=\frac{n!}{p!q!(n-i)!}.
\label{eq:sizeD}
\end{eqnarray*}
By Lagrange's
Theorem, the cardinality of $D_{i,p}$ is equal to the size of the
$W(\ddC_n)$-orbit of $B_{i,p}$.
Therefore by Lemma \ref{lem:numip},  $N_{i,p}$
is the stabilizer of $B_{i,p}$ in $W(\ddC_n)$.
\end{proof}

The study of the stabilizer of $B_{i,p}$ will now be used to rewrite
products of
$b_{p,i,p'}$ with elements of $W(\ddC_n)$. The result is in Lemma \ref{anyr}
and needs the following special cases.

\begin{lemma} \label{lm:rfp}
Let $i\in\{0,\ldots,n\}$ and $p\in\{0,\ldots,i\}$ with $i-p$ even.
\begin{enumerate}[(i)]
\item For each  $r\in A_{i,p}$ we have
$r  b_{p,i,i} = b_{p,i,i}$.
\item For each  $v\in L_i$
we have
$ve_{i,p}=e_{i,p}v$
 and
$v b_{p,i,i}= b_{p,i,i} v$.
\end{enumerate}
\end{lemma}

\begin{proof}
(ii). By Lemma \ref{X_n} and Definition \ref{0.1},
the two equations hold for the generators of $L_i$.
Therefore they hold for each element of $L_i$.

\nl(i).  The roots $\beta_{j}$ and $r_{j}r_{j-1}\cdots r_{1}\beta_0$ are as
in Proposition \ref{lem:rel}(iii), so
\begin{eqnarray*}
r_{j}z_{j}z_{j+1}&=& (r_{j}z_{j}r_j z_j)
r_j\overset{(\ref{0.1.19})}{=}z_{j}(e_j z_j r_j)
\overset{(\ref{0.1.18})}{=} z_{j}e_j z_j \overset{(\ref{4.1.5})}{=}z_{j}r_j z_j r_j=z_{j}z_{j+1}.
\end{eqnarray*}
This proves that (i) is satisfied with $r = r_j$ for $j=0,\ldots, p-1 $.
For the choices $r = r_{p+2k-1}$ for $k = 1,\ldots,q$ this is straightforward.
Moreover,
$t_{0,p} e_{p+1}
=y_{p+1}(r_{p+1}y_{p+1}e_{p+1})=y_{p+1}y_{p+1}e_{p+1}=e_{p+1}$, and
\begin{eqnarray*}
t_{k,p}e_{p+2k-1}e_{p+2k+1}
&=&r_{p+2k}r_{p+2k+1}r_{p+2k-1}r_{p+2k}e_{p+2k-1}e_{p+2k+1}\\
&\overset{(\ref{0.1.13})}{=}&(r_{p+2k}r_{p+2k+1}e_{p+2k})e_{p+2k-1}e_{p+2k+1}\\
&\overset{(\ref{0.1.13})+(\ref{0.1.9})}{=}&(e_{p+2k+1}e_{p+2k}e_{p+2k+1})e_{p+2k-1}\\
&\overset{(\ref{3.1.5})}{=}&e_{p+2k-1}e_{p+2k+1}.
\end{eqnarray*}
So (i) holds for all generators of $A_{i,p}$ and hence for
all of $A_{i,p}$.
\end{proof}

\begin{lemma} \label{anyr}
Suppose $r\in W(\ddC_n)$.
Let $i\in\{0,\ldots,n\}$ and $p\in\{0,\ldots,i\}$ with $i-p$ even.
\begin{enumerate}[(i)]
\item There are $u\in D_{i,p}$ and $v\in L_i$
such that
$r b_{p,i,i}=u b_{p,i,i}v$.
\item There are $u'\in D_{i,p}^{\op}$ and $v'\in L_i$ such that
$b_{i,i,p}r=v'b_{i,i,p}u'$.
\end{enumerate}
\end{lemma}

\begin{proof}
Let $r\in W(\ddC_n)$. By Lemma \ref{lem:Li} and Definition
\ref{df:Ai} for $D_{i,p}$,
there exist
$u\in D_{i,p}$, $v\in L_i$, and $a\in A_{i,p}$ such that $r=uva$.
By Lemma \ref{lm:rfp},
\begin{eqnarray*}r b_{p,i,i} &=&u v a  b_{p,i,i}
                = u v    b_{p,i,i}
                =u   b_{p,i,i}v.
\end{eqnarray*}
The second statement follows by applying Proposition \ref{prop:opp} to (i).
\end{proof}

Our next step towards a normal form for elements of $\BrM(\ddC_n)$ is to
describe products of elements from $W(\ddC_n)b_{p,i,p'}W(\ddC_n)$ with
generators $e_j$. To this end we first prove two useful equalities.

\begin{lemma}\label{3fs}
In $\Br(\ddC_{n})$, the following hold for $i\in\{2,\ldots,n-1\}$.
\begin{eqnarray}
e_iz_{i+1}&=&e_i z_{i},\label{ezez1}\\
e_{i-1}z_{i+1}z_{i}z_{i-1}&=&r_{i}r_{i-1}e_{i}z_{i}z_{i+1}z_{i-1},
\label{ezzz}\\
e_{i}z_{i+1}z_{i}e_{i}&=&\delta^2 e_{i}.
\label{ezze}
\end{eqnarray}
\end{lemma}

\begin{proof}
By Lemma \ref{lem:rel}(iii),
$e_iz_{i+1}\overset{(\ref{df:zn})}{=}e_i
r_iz_{i}r_i\overset{(\ref{0.1.4})}{=} e_i z_{i}r_i
\overset{(\ref{0.1.18})}{=}e_i z_{i}$. This proves (\ref{ezez1}).

Now (\ref{ezze}) follows from Lemma \ref{lem:rel}(i) as
\begin{eqnarray*}
e_{i}z_{i+1}z_{i}e_{i}&\overset{(\ref{ezez1})}{=}&e_{i}z_{i}^2e_{i}
=\delta e_{i}z_{i}e_{i}
\overset{(\ref{0.1.16})}{=}\delta^2  e_{i}.
\end{eqnarray*}

As for (\ref{ezzz}),
note that (\ref{0.1.19}) and Proposition \ref{lem:rel}(iii) give
$z_{i-1}e_{i-1}z_{i-1} = r_{i-1} z_{i-1}r_{i-1} z_{i-1}=z_i z_{i-1}$.
Hence
\begin{eqnarray*}
r_{i-1}r_{i}e_{i-1}z_{i+1}z_{i}z_{i-1}&\overset{(\ref{0.1.13})}{=}&
e_ie_{i-1}z_{i+1}z_i z_{i-1}
\overset{\ref{X_n}}{=}
e_iz_{i+1} e_{i-1}z_{i}z_{i-1}
\overset{\rm{(\ref{ezez1})}}{=}
\delta e_iz_{i} e_{i-1}z_{i-1}\\
&\overset{\rm{(\ref{ezez1})}+\ref{prop:opp}}{=}&
\delta e_iz_{i-1} e_{i-1}z_{i-1}
= \delta e_i z_{i}z_{i-1}
\overset{\rm{(\ref{ezez1})}}{=}
 e_i z_{i+1} z_{i}z_{i-1},
\end{eqnarray*}
and the equation follows by left multiplication with $r_ir_{i-1}$.
\end{proof}

The detailed information stated in the last sentence of the following
proposition will be needed for the proof of cellularity of
$\Br(\ddC_n,R,\delta)$ in the next section.

\begin{prop}\label{lm:efp}
Let $i,p,p'$ be natural numbers with $0\le p,p'\le i \le n$ and $i-p$ and
$i-p'$ even.  For each root $\beta\in \Psi^+$, there are
$h\in\{i,i+1,i+2\}$, $k$, $m$, $m'\in\N$, $u\in D_{h,m}$,
$w\in D_{h,m'}^{\op}$, and $v\in L_h$ such
that
$$e_{\beta}b_{p,i,p'} = \delta^k u b_{m,h,m'} vw .$$
Moreover,
if $h=i$, then $w=1$ and $m'=p'$, while $k$, $u$, and $v$ do not depend on
$p'$.
\end{prop}

\begin{proof}
We first sketch the general idea of proof. There are only two possible root
lengths in the Coxeter root system $\Psi$ of type $\ddC_n$.  We call
$\beta\in\Psi$ short if $(\beta,\beta) = 1$ and long otherwise, in which
case $(\beta,\beta) = 2$. In each case, only one root needs to be
considered, for all other roots of the same length are conjugate to this
particular one under the natural action of $W(\ddC_n)$, and
Lemma \ref{anyr} can
be applied to reduce to the representative root.

The top of $b_{p,i,p'}$ is the admissible set $B_{i,p}$ displayed in Figure
\ref{fig:Bip}.

First suppose $\beta$ is long.  Then $\beta$ can be written as
$\beta_0+2\beta_1+\cdots+2\beta_{t-1}$, for some $1\leq t\leq n$, and so
$e_{\beta}=z_{t}$.  We will distinguish cases according to relations among
$t$, $i$, and $p$, and apply induction on $t$ and $i$.
In Figure \ref{fig:threestrands} the roots of the Cases L1, L2,
and L3 are displayed as the top, middle, and bottom horizontal strand,
respectively.

\begin{figure}
\begin{center}
\includegraphics[width=.9\textwidth,height=.1\textheight]{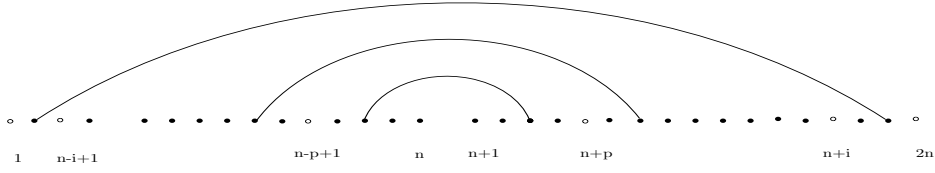}
\end{center}
\caption{Horizontal strands representing the three
cases for a long root $\beta$.}
\label{fig:threestrands}
\end{figure}

\nl
{\bf Case L1.}
Suppose $t\geq i+1$.
For any $t>i+1$,
we have $z_t=sz_{i+1}s^{-1}$,
where $s=r_{t-1}\cdots r_{i+1}$.
This implies that
$z_t b_{p,i,p'}=s z_{i+1}b_{ p,i,p'}s^{-1}$, with $s\in L_i$.
Lemmas \ref{anyr} and \ref{lm:rfp} can be used to reduce this case to the case where
$t=i+1$.

If $p=p'=i$, then $e_{i,p}=e_{i,p'}=1$ and $z_{i+1}b_i=b_{i+1}$ by
the definition of $b_{i}$ in (\ref{eq:bi}), as required.

If $p\neq i= p'$, it suffices to prove the equation
$$z_{i+1}b_{p,i,i}=r b_{p+1,i+1,i+1},$$
with $r$ in the subgroup of $W(\ddC_{n})$ generated by $r_0, r_1, \ldots, r_i$.
We proceed by induction on $i$.
In view of Lemma \ref{3fs}, (below IH is short for Inductive Hypothesis),
\begin{eqnarray*}
z_{i+1}b_{p,i,i}&\overset{(\ref{eq:bi})+(\ref{fip})}{=}&z_{i+1}e_{i-1}e_{i-2,p}z_{i}z_{i-1} b_{i-2}
\overset{\ref{X_n}}{=}e_{i-1}z_{i+1}z_{i}z_{i-1}e_{i-2,p} b_{i-2}\\
&\overset{(\ref{ezzz})}{=}&r_{i}r_{i-1}e_{i}z_{i}z_{i+1}z_{i-1}e_{i-2, p}
b_{i-2}
\overset{\ref{X_n}}{=}r_ir_{i-1}e_{i}(z_{i-1}b_{p,i-2,i-2})z_{i}z_{i+1}\\
&\overset{{\rm IH}}{=}& r_ir_{i-1}e_{i}gb_{p+1,i-1,i-1}z_{i}z_{i+1}
=r_ir_{i-1}ge_{i} b_{p+1,i-1,i-1}z_{i}z_{i+1}\\
&=&r_ir_{i-1}gb_{p+1,i+1,i+1},
\end{eqnarray*}
where $g$ is an element of the subgroup of $W(\ddC_n)$
generated by $r_0, r_1, \ldots, r_{i-2}$.
Hence the claim holds.

The case $p=i\neq p'$ now follows by use of Proposition \ref{prop:opp}.

If $p$, $p'\neq i$ then, by the above,
\begin{eqnarray*}
z_{i+1}b_{p,i,p'}&=&r e_{i+1,p+1}b_{i+1}e_{i,p'}
=r e_{i+1,p+1}b_ie_{i,p'}z_{i+1}
=r e_{i+1,p+1}b_{i+1}e_{i+1,p'+1}r'\\
 &=&r b_{p+1,i+1,p'+1}r',
\end{eqnarray*}
where $r$, $r'\in W(\ddC_n)$. By Lemma \ref{lm:rfp}, we conclude that this
expression can be written in the required form with $h=i+1$.

\nl
{\bf Case L2.} Next suppose $p+1\leq t<i+1$.
By definition of $z_i$,

\begin{eqnarray*} z_{i-1}&=&r_{i-1}z_ir_{i-1},\\
                         z_{i-2}&=&r_{i-2}r_{i-3}r_{i-1}r_{i-2}z_ir_{i-2}r_{i-3}r_{i-1}r_{i-2},
                           \end{eqnarray*}
                           with $r_{i-1}$, $r_{i-2}r_{i-3}r_{i-1}r_{i-2}\in A_{i,p}$. By induction on $t$, we can find $r\in A_{i,p}$
                           such that $z_t=rz_ir^{-1}$. By Lemma \ref{lm:rfp},
                           $$e_{\beta}b_{p,i,p'}=z_tb_{p,i,p'}=
rz_ir^{-1}b_{p,i,p'}=rz_ib_{p,i,p'}.$$
In view of Lemma \ref{anyr},
this reduces the problem to rewriting
$z_ib_{p,i,p'}$ in the required form.

Now
\begin{eqnarray*}
z_i e_{i-1}z_i z_{i-1}&=&(r_{i-1}z_{i-1}r_{i-1}e_{i-1})z_{i-1}z_{i}
\overset{(\ref{0.1.4})+(\ref{4.1.2})}{=}z_{i-1}e_{i-1}z_{i-1} z_{i}\\
&\overset{(\ref{0.1.19})}{=}&(r_{i-1}z_{i-1}r_{i-1})z_{i-1}z_{i}
=z_{i}z_{i-1}z_{i}\\
 &=&\delta z_{i}z_{i-1},
\end{eqnarray*}
so \begin{eqnarray*}
z_ib_{p,i,i}&\overset{(\ref{fip})}{=}&z_i e_{i-1}e_{i-2,p}
z_{i-1}z_{i}b_{i-2}
\overset{\ref{X_n}}{=}z_i e_{i-1} z_{i-1}z_{i}e_{i-2,p}b_{i-2}\\
 &=&\delta z_{i-1}z_{i}b_{p,i-2,i-2},\\
\end{eqnarray*}
by the claim of Case L1 and Lemma \ref{anyr}, the above
can be written as $ub_{p+2,i,i}v$ with $u\in D_{i,p+2}$ and $v\in L_i$,
and so the proposition holds in this case
as $z_ib_{p,i,p'}=ub_{p+2,i,p'}v$.

\nl
{\bf Case L3.} We remain with the case where  $1\leq t\leq p$. We have
\begin{eqnarray*}
z_tb_{p,i,p'} &=&
e_{i,p}
z_tb_ie_{i,p'}
=\delta e_{i,p} b_ie_{i,p'}=\delta b_{p,i,p'},
\end{eqnarray*} and so the proposition holds
with $h=i$.

We next  consider the case where $\beta$ is a short root,
which means $\beta=\beta_{s}+\beta_{s+1}+\cdots+\beta_{t}$ with $0< s\leq
t\leq n-1$ or
$\beta=\beta_{0}+2\beta_1+\cdots+2\beta_{s-1}+\beta_{s}+\beta_{s+1}+\cdots+\beta_{t}$
with $0\leq s\leq t\leq n-1$.  We will distinguish seven cases by values of
$s$ and $t$ corresponding to the horizontal strands of Figure
\ref{fig:sevenstrands}.
The seven cases occur in the order from top to bottom.

\begin{figure}
\begin{center}
\includegraphics[width=.9\textwidth,height=.7\textheight]{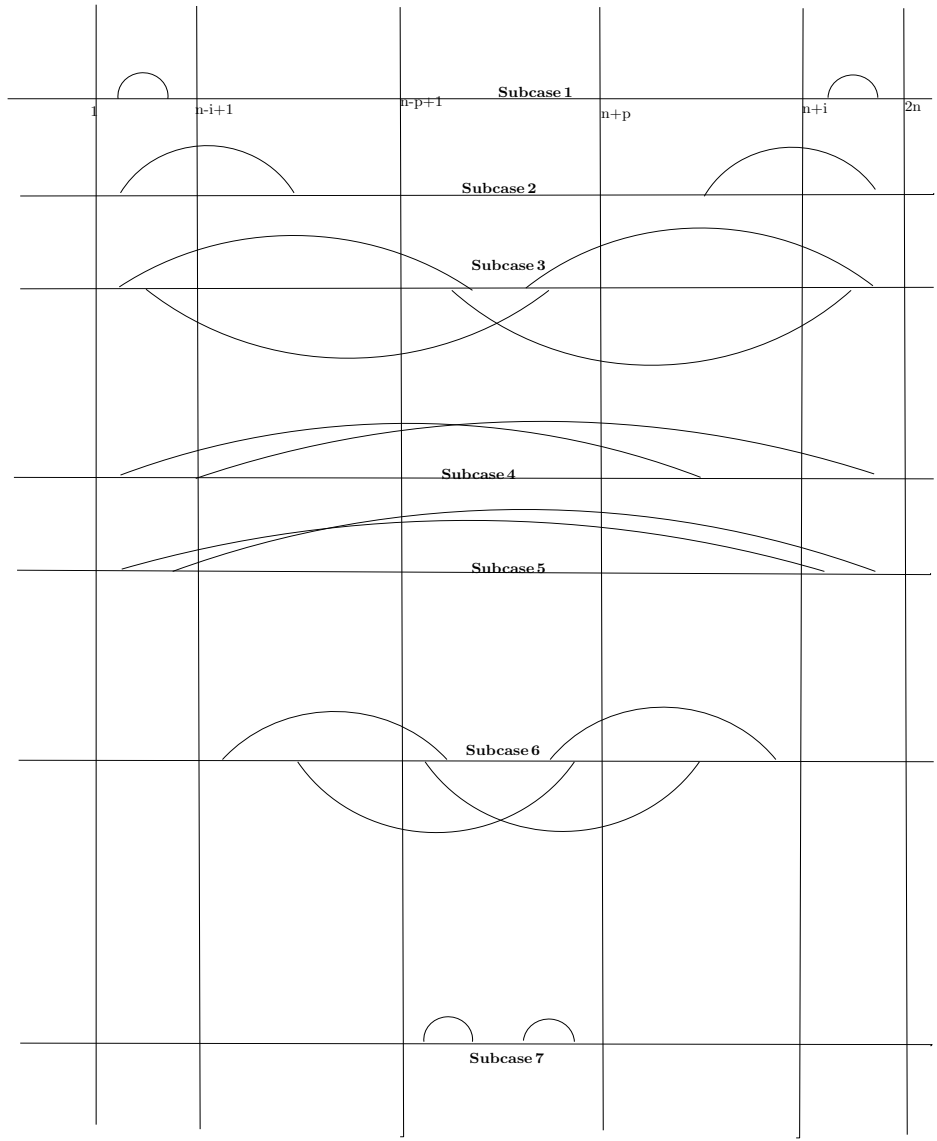}
\end{center}
\caption{Strands corresponding to 7 possibilities for the long root $\beta$.}
\label{fig:sevenstrands}
\end{figure}

\nl {\bf Case S1.} Suppose that $\beta$ is a linear combination of
$\beta_j$ $(j=i+1,\ldots,n-1)$.  Then, by Lemma \ref{3fs},
\begin{eqnarray*}
e_{i+1}b_{p,i,p'}
&=&
e_{i,p}e_{i+1}b_ie_{i,p'}
\overset{(\ref{ezze})}{=}\delta^{-2}e_{i,p}
(e_{i+1}z_{i+2}z_{i+1}e_{i+1})b_ie_{i,p'}\\
&=&\delta^{-2}e_{i,p}e_{i+1}z_{i+2}z_{i+1}b_ie_{i+1}e_{i,p'}\\
&\overset{(\ref{eq:bi})+(\ref{fip})}{=}&\delta^{-2}e_{i+2,p}b_{i+2}e_{i+2,p'}
= \delta^{-2}b_{p,i+2,p'}.
\end{eqnarray*}
At the same time, for any such $\beta\in\Psi^+$, there exists an element
$r\in L_i$ such that $\beta=r \beta_{i+1}$, thus $e_{\beta}=re_{i+1}r^{-1}$.
Now $e_{\beta}b_{p,i,p'}=\delta^{-2}r
b_{p,i+2,p'}r^{-1}$, and hence the proposition holds with $h=i+2$.

\nl
{\bf Case S2.}
Suppose
$\beta=\beta_{s}+\beta_{s+1}+\cdots+\beta_t$, with $p\leq s\leq i\leq t\leq n-1$.

If $p=i$, then $e_{i,p}=1$. First, consider the case $t=s=i$.  Since
\begin{eqnarray*} e_iz_{i+1}z_{i}&\overset{(\ref{ezez1})}{=}&
e_i z_{i}z_i\overset{\ref{lem:rel}(i)}{=}\delta e_i z_i,
\end{eqnarray*}
the use of the claim of Case L1 and Proposition \ref{prop:opp}
gives the existence of an element $r\in W(\ddC_n)$ such that
\begin{eqnarray*}
e_{i}b_{i,i,p'}&=&\delta^{-1}e_iz_ib_ie_{i,p'}
\overset{\ref{lem:rel}(i)}{=}\delta^{-1} e_iz_{i+1}b_ie_{i,p'} 
\overset{\ref{X_n}}{=}\delta^{-1} e_ib_{i,i,p'}z_{i+1}\\
&\overset{{\rm L1}}{=}&\delta^{-1} e_ib_{i+1,i+1,p'}r
=\delta^{-1} b_{i-1,i+1,p'}r,
\end{eqnarray*}
as required.

If $t\neq s$, then
$e_\beta=r_{\beta''}r_{\beta'}e_ir_{\beta'}r_{\beta''}$, where
$\beta'=\beta_s+\cdots+\beta_{i-1}$, $\beta''=\beta_{i+1}+\cdots+\beta_{t}$,
which implies that $r_{\beta'}\in A_{i,p}$, $r_{\beta''}\in L_i$, and hence
$e_\beta b_ie_{i,p'}=r_{\beta''}r_{\beta'}e_i b_ie_{i,p'}r_{\beta''}$.
As in the argument for $e_ib_{i,i,p'}$ above, this can be
written in the required form with $h=i+1$.

On the other hand if $p\neq i$ , then $e_{i,p} \neq 1$. Therefore for
$\beta=\beta_{s}+\beta_{s+1}+\cdots+\beta_t$, with $p\leq s\leq i\leq t\leq
n-1$, since there is some $l\in \{p+1, p+3,\ldots, i-1\}$, such that
$e_\beta e_l=r_l r_{\beta}e_l$, implying that $e_{\beta}e_{i,p}=r_l
r_{\beta}e_{i,p}$ and $e_{\beta}b_{p,i,p'}=r_l r_{\beta}b_{p,i,p'}$. By
Lemma \ref{anyr}, the proposition holds with $h=i$.

\nl
{\bf Case S3.}
Suppose $\beta=\beta_{s}+\beta_{s+1}+\cdots+\beta_t$
or  $\beta=\beta_{0}+2\beta_1+\cdots+2\beta_{s-1}+\beta_{s}+\cdots+\beta_t$ with $0<s\leq p$ and $i\leq t\leq n-1$.

First consider  $\beta=\beta_{p}+\cdots +\beta_{i}$. Following the argument for
$e_iz_{i+1}z_i=\delta e_i z_i$ in the above case,
we find that $e_\beta z_{i+1}z_i=\delta e_\beta z_i$ holds,
which implies
\begin{eqnarray}\label{subcase3}
e_{\beta}b_{p,i,p'}=e_{i,p}e_{\beta}b_ie_{i,p'}=\delta^{-1}e_{i,p}e_{\beta}b_{i+1}e_{i,p'}.
\end{eqnarray}
Observe that
\begin{eqnarray*}
r_{i-1}r_{i}e_{\beta}e_{i-1}&=&
e_{\beta-\beta_{i}-\beta_{i-1}}r_{i-1}r_{i}e_{i-1}r_i r_{i-1}r_{i-1}r_{i}
\overset{(\ref{0.1.15})}{=}
e_{\beta-\beta_{i}-\beta_{i-1}}e_{i}r_{i-1}r_{i},\\
\mbox{and }\, r_{i-1}r_{i}b_{i+1}&\overset{\ref{lm:rfp}}{=}&b_{i+1}.
\end{eqnarray*}
Therefore (\ref{subcase3}) can be written as
\begin{eqnarray*}\delta^{-1}e_{i,p}e_{\beta}b_{i+1}e_{i,p'}
&=&\delta^{-1}(e_{i-1}r_{i}r_{i-1}) e_{i-2,p}e_{\beta-\beta_{i}-\beta_{i-1}}b_{i+1}e_{i,p'}\\
&\overset{(\ref{0.1.15})}{=}&\delta^{-1}r_{i}r_{i-1}e_{i} e_{i-2,p}e_{\beta-\beta_{i}-\beta_{i-1}}b_{i+1}e_{i,p'}\\
&=&\delta^{-1}r_{i}r_{i-1}e_{i} (e_{\beta-\beta_{i}-\beta_{i-1}}e_{i-2,p}b_{i-2})z_{i-1}z_{i}z_{i+1}e_{i,p'}.
\end{eqnarray*}
By induction on $i$, we can use an argument  as in the claim of Case S1,
and the  above can be written as
\begin{eqnarray*}\delta^{-1}r_{i}r_{i-1}e_{i}
ge_{i-1,p-1}b_{i-1}z_{i-1}z_{i}z_{i+1}e_{i,p'}
&=&r_{i}r_{i-1}ge_{i} e_{i-1,p-1}b_{i-1} z_{i}z_{i+1}e_{i,p'}\\
&=&r_{i}r_{i-1}g e_{i+1,p-1}b_iz_{i+1}e_{i,p'}\\
&=&r_{i}r_{i-1}g e_{i+1,p-1}b_{i,i,p'}z_{i+1},
\end{eqnarray*}
where $g\in W(\ddC_n)$ is a
product of elements from $r_0, r_1, \ldots, r_{i-2}$.  By
Case L1 and Proposition \ref{prop:opp}, the proposition
holds with $h=i+1$.

We return to the general setting of Case S3.  Then there exists $r'\in
L_{i}$ and $r'' \in A_{i,p}$ such that $r' r'' \hat\beta=\beta$, with
$\hat\beta=\beta_{p}+\cdots +\beta_{i}$.  Then $$e_{\beta}b_{p,i,p'}=r' r''
e_{\hat\beta}b_{p,i,p'} r' ,$$ hence the proposition holds in Case S3 due to Lemma \ref{anyr}.

\nl {\bf Case S4.}  If
$\beta=\beta_{0}+2\beta_1+\cdots+2\beta_{s-1}+\beta_{s}+\cdots+\beta_t$ with
$p\leq s\leq i\leq t\leq n-1$, and let $\beta'=\beta_{s}+\cdots+\beta_t$. Then $e_{\beta}=y_{s}e_{\beta'}y_{s}.$ When $p=i=s$, we see that $y_s\in
A_{i,p}$; when $p\neq i$, there is some $l\in \{p+1, p+3,\ldots, i-1\}$,
such that $e_\beta e_l=r_l r_{\beta}e_l$.  This brings us back to the
argument of Case S2.

\nl
{\bf Case S5.}
If $\beta=\beta_{0}+2\beta_1+\cdots+2\beta_{s}+\beta_{s+1}+\cdots+\beta_t$, with $i<
s\leq t\leq n-1$, then $\beta=y_{s+1}\beta'$, with
$\beta'=\beta_{s+1}+\beta_{s+2}+\cdots+\beta_t$, and  $y_{s+1}\in L_i$.
Therefore
$$e_{\beta}b_{p,i,p'}=y_{s+1}(e_{\beta '}b_{p,i,p'})y_{s+1},$$
and we are back in Case S1. 

\nl {\bf Case S6.}  If
$\beta=\beta_{0}+2\beta_1+\cdots+2\beta_{s}+\beta_{s+1}+\cdots+\beta_t$ or
$\beta=\beta_{s}+\beta_{s+1}+\cdots+\beta_t$, with $0\leq s\leq t$ and $p
\leq t\leq i-1$, then there must be some $e_{j}\in \{
e_{p+2j-1}\}_{j=1}^{(i-p)/2}$ such that $\beta$ is not orthogonal to
$\beta_j$, or $\{\beta, \beta_j\}$ is not admissible, or $\beta=\beta_j$.
Then, by Lemma \ref{lem:rel}, there exists $r\in W(\ddC_n)$ such
that $e_\beta e_{i,p}=re_{i,p}$ or $e_\beta e_{i,p}=\delta re_{i,p}$.
This implies that the proposition holds with $h=i$.

\nl {\bf Case S7.}
If $\beta$ can be written as a linear combination of $\{\beta_j\}_{j=0}^{p-1}$, then
$\beta$ is conjugate to $\beta_{p-1}$ under the subgroup of $A_{i,p}$ generated by $\{r_j\}_{j=0}^{p-1}$.
 Then we can find a $r\in A_{i,p}$ such that $r \beta_{p-1}=\beta$,
which implies  
\begin{eqnarray*}
e_{\beta}b_{p,i,p'}&=&
r e_{p-1}r^{-1} b_{i,i,p'}
\overset{\ref{lm:rfp}(i)}{=}r e_{p-1} b_{i,i,p'}
=r b_{p-2,i,p'},
\end{eqnarray*}
so the proposition holds
with $h=i$
due to Lemma \ref{anyr}.
\end{proof}

\begin{thm} \label{th:uniqueform}
Each element in the monoid $\BrM(\ddC_n)$ can be written as
$$\delta^k u b_{p,i,p'} v  w^{\op},$$
where
$k\in \Z$ and  $i,p,p'\in\{0,\ldots, n\}$ with $i-p$ and
$i-p'$ even,
$u\in D_{i,p}$, $v\in L_i$,  and $w\in D_{i,p'}$. In
particular, $\Br(\ddC_{n})$ is free of rank at most $a_{2n}$.
\end{thm}

\begin{proof} Let $U$ be the set of elements of
$\BrM(\ddC_n)$ of the indicated form. We show that $U$ is invariant under
left multiplication by generators of
$\BrM(\ddC_n)$. To this end, consider an arbitrary
element $a = \delta^k ub_{p,i,p'} v  w^{\op}$ of $U$.
Obviously $\delta^{\pm1}a\in U$.
Without loss of generality, we may take $k=0$.

Let $r\in W(\ddC_n)$.  By Lemma \ref{anyr} applied to $ru$ there are $u'\in
D_{i,p}$ and $v'\in L_i$ such that $ra =  rub_{p,i,p'}
v w^{\op}$.  By Lemma \ref{lm:rfp}, this is equal to $  u'
b_{p,i,p'}v'v w^{\op}$ and, as $v'v\in L_i$, the set $U$ is
invariant under left multiplication by Weyl group elements.

Finally, consider the generator $e_j$ of $\BrM(\ddC_n)$.
Writing $\beta = u^{\op}\alp_j$ we have
$e_ja =  e_jub_{p,i,p'}vw^{\op}=
 u e_{\beta}b_{p,i,p'}vw^{\op}$
and by Proposition \ref{lm:efp}
this belongs to $U$ again.

Now, by Proposition \ref{prop:opp}
we also find that $U$ is invariant under right multiplication by generators.
This proves that $U$ is invariant under both left and right multiplication
by any generator of $\BrM(\ddC_n)$.
As it contains the identity $(b_{0,0,0})$, it follows that $U$ coincides with
the whole monoid.

As for the last assertion of the theorem, observe that freeness of
$\Br(\ddC_n)$ over $\Z[\delta^{\pm1}]$ is immediate from the fact that it is
a monomial algebra with a finite number of generators.
By the first assertion, its rank
is at most
$$\sum_{i=0}^n
\left(\sum_{p\equiv i\pmod2} |D_{i,p}|\right)^2\cdot |L_i|.$$
By Lemma \ref{lem:numip},
the cardinality of  $D_{i,p}$ is $n!/(p!q!(n-i)!)$, where $q =
(i-p)/2$,
and,
by Lemma \ref{lem:Li},
$L_i$ is isomorphic to $W(\ddC_{n-i})$, which has
$2^{n-i}(n-i)!$ elements.
Therefore, the rank of
$\Br(\ddC_n)$ over $\Z[\delta^{\pm1}]$ is at most
$$
\sum_i \left(\sum_{p,q: p+2q=i}
\frac{n!}{p!q!(n-i)!}\right)^2  2^{n-i}(n-i)! .
$$
Corollary \ref{lemma:anformproof} gives that this sum is equal to $a_{2n}$.
\end{proof}

\bigskip
\label{sect:phi}
We are now ready to prove Theorem \ref{thm:main}.
Theorem
\ref{th:uniqueform} shows that $\Br(\ddC_{n})$ is free of rank at most
$a_{2n}$.
By Proposition \ref{prop:phisurj},
the homomorphism
$\phi: \Br(\ddC_n) \to \SBr(\ddA_{2n-1})$
is surjective, so the rank of
$\Br(\ddC_n)$
is at least the rank of
$\SBr(\ddA_{2n-1})$, which is known to be $a_{2n}$ by
Corollary \ref{cor:3.4}.
Thus, the ranks of
$\Br(\ddC_n)$ and $\SBr(\ddA_{2n-1})$ coincide and
$\phi$ is an isomorphism.
\end{section}

\begin{section}{Further properties of type $\ddC$ algebras}
\label{sect:cellular}

In this section we prove that the algebra $\BrM(\ddC_n,R,\delta)$ is
cellular, in the sense of Graham and Lehrer \cite{GL1996}, provided $R$ is an integral domain containing the inverse to
$2$. The proof given here runs parallel to the proof of the
corresponding result for $\ddD_n$ in \cite[Section~6]{CGW2006}.  We finish
by discussing a few more desirable properties of the newly found Brauer
algebras.

Recall from \cite{GL1996} that an associative
algebra $\alg$ over a commutative ring $R$ is cellular if there is a quadruple
$(\Lambda, T, C, *)$ satisfying the following three conditions.

\begin{itemize}
\item[(C1)] $\Lambda$ is a finite partially ordered set.  Associated to each
$\lambda \in \Lambda$, there is a finite set $T(\lambda)$.  Also, $C$ is an
injective map
$$ \coprod_{\lambda\in \Lambda} T(\lambda)\times T(\lambda) \rightarrow \alg$$
whose image is an $R$-basis of $\alg$.

\item[(C2)]
The map $*:\alg\rightarrow \alg$ is an
$R$-linear anti-involution such that
$C(x,y)^*=C(y,x)$ whenever $x,y\in
T(\lambda)$ for some $\lambda\in \Lambda$.

\item[(C3)] If $\lambda \in \Lambda$ and $x,y\in T(\lambda)$, then, for any
element $a\in \alg$,
$$aC(x,y) \equiv \sum_{u\in T(\lambda)} r_a(u,x)C(u,y) \
\ \ {\rm mod} \ \alg_{<\lambda},$$ where $r_a(u,x)\in R$ is independent of $y$
and where $\alg_{<\lambda}$ is the $R$-submodule of $\alg$ spanned by $\{
C(x',y')\mid x',y'\in T(\mu)\mbox{ for } \mu <\lambda\}$.
\end{itemize}
Such a quadruple $(\Lambda, T, C, *)$ is called a {\em cell datum} for
$\alg$.  We will describe such a quadruple for $\Br(\ddC_n,R,\delta)$.

For $*$ we will use the
anti-involution $\op$ determined in Proposition \ref{prop:opp}.
Let $i \in\{0,\ldots,n\}$.
By Theorem \ref{th:uniqueform}, each element in the monoid
$\BrM(\ddC_n)$
can be written in the form
$$\delta^k u b_{p,i,p'} v w^{\op},$$ where $k\in \Z$ and
$i,p,p'\in\{0,\ldots, n\}$ are such that $i-p$ and $i-p'$ are even, $u\in
D_{i,p}$, $v\in L_i$, and $w\in D_{i,p'}$. As the coefficient ring $R$
is an integral domain containing the inverse of $2$,
it satisfies the conditions of \cite[Theorem 1.1]{G2007}, so by
\cite[Corollary 3.2]{G2007} the group rings $R[L_i]$ $(i=0,\ldots,n)$
are all
cellular. By Lemma \ref{lem:Li},
the subalgebra $RL_i$ of $\Br(\ddC_n,R,\delta)$ (with unit $\hb^{(i)}$,
see (\ref{eq:ehat}))
generated by $L_i$ is isomorphic to $R[L_i]$.
Let $(\Lambda_{i},T_i, C_i, *_{i})$ be a cell datum for $RL_i$ an in
\cite{G2007}.  Observe that the generators of $L_i$ in Definition
\ref{df:Ai} are fixed by $\op$, so $RL_i$ is $\op$-invariant.  By
\cite[Section 3]{G2007}, $*_{i}$ is the map $\op$ on $RL_i$ and so $*_{i}$
is the restriction of $\op$ to $RL_i$.

The underlying poset $\Lambda$ will be $\{B_i\}_{i=0}^{n}$, as defined in (\ref{eq:Bi}). We
say that $B_i>B_j$ if and only if $i<j$ or, equivalently, $B_i\subset B_j$.
In particular, $\emptyset$ is the greatest element of $\Lambda$.

The set $T(B_i)$ is taken to be the set of all triples $(u,e_{i,p},s)$ where
$u\in D_{i,p}$ (see Definition \ref{df:Ai}), $p \in \{0, \ldots, i\}$ with $i - p$ even, and the product $e_{i,p}$ is given by (\ref{fip}), and $s\in T_i$.  Clearly, this set is finite.

The map $C$ is given by
$C((u,e_{i,p},s),(w,e_{i,p'},t))=ue_{i,p}C_i(s,t)e_{i,p'}w^{\op}$.  By Lemma
\ref{lm:rfp}, $ub_{p,i,p'}vw^{\op} = (ue_{i,p})(b_iv)(we_{i,p'})^{\op}$, so
the image of $C$ is a basis by Theorems \ref{thm:main} and
\ref{th:uniqueform}, and the fact that $\{C_i(s,t)\mid s, t\in T_i\}$ is a
basis for $RL_i$ (which is a consequence of (C1) for $(\Lambda_{i},T_i, C_i,
*_{i})$).  This gives a quadruple $(\Lambda, T, C, *)$ satisfying (C1).

For (C2) notice that
$(ue_{i,p}C_i(s,t)e_{i,p'}w^{\op})^{{\op}} =
we_{i,p'}C_i(s,t)^{{\op}}e_{i,p}u^{\op}$.
Now $C_i(s,t)^{{\op}}=C_i(t,s)$, by the cellularity condition (C2) for
$RL_i$ and so (C2) holds for the cell datum $(\Lambda, T, C, *)$.

Finally, we check condition (C3) for $(\Lambda, T, C, *)$. It suffices to
consider the left multiplications by $r_j$ and $e_j$ of
$ue_{i,p}C_i(s,t)e_{i,p'}w^{\op} $.  Up to linear combinations, we can
replace the latter expression by $ue_{i,p}b_ive_{i,p'}w^{\op} $ for $v\in
L_i$.  Now, by Lemma \ref{anyr} for $r_j$ and Proposition \ref{lm:efp} 
for $e_j$ (note the product lies in
$\Br(\ddC_n,R,\delta)_{< B_i}$ if $h> i$), (C3) holds for the cell datum
$(\Lambda, T, C, *)$.  Therefore we have now proved

\begin{thm}\label{th:cellular}
Let $R$ be an integral domain with $2^{-1}\in R$.  Then the quadruple
$(\Lambda, T, C, *)$ is a cell datum for $\Br(\ddC_n,R,\delta)$, proving
the algebra is cellular.
\end{thm}

We continue by discussing some desirable properties of the Brauer algebra
$\Br(\ddC_n)$. First of all, for any disjoint union $M$ of diagrams of type
$Q$, for $Q$ a simply laced graph, and $\ddC_k$ for $k\in\N$, the Brauer
algebra is defined as the direct product of the Brauer algebras whose types
are the components of $X$. The next result states that parabolic subalgebras
behave well.

\begin{prop}\label{prop:parabolic}
Let $J$ be a set of nodes of the Dynkin diagram $\ddC_n$.
Then the parabolic subalgebra of the Brauer algebra $\Br(\ddC_n)$,
that is, the subalgebra generated by $\{r_j,e_j\}_{j\in J}$,
is isomorphic to the Brauer algebra of type $J$.
\end{prop}

\begin{proof}
In view of induction on $n-|J|$ and restriction to connected components of
$J$, it suffices to prove the result for $J = \{1,\ldots,n-1\}$ and for
$J=\{0,\ldots,n-2\}$. In the former case, the type is $\ddA_{n-1}$ and the
statement follows from the observation that the symmetric diagrams without
strands crossing the vertical line through the middle of the segments
connecting the dots $(n,1)$ and $(n+1,1)$ are equal in number to the Brauer
diagrams on the $2n$ nodes (realized to the left
of the vertical line). In the latter case, the type is $\ddC_{n-1}$ and the
statement follows from the observation that the symmetric diagrams with
vertical strands from $(1,1)$ to $(1,0)$ and from $(2n,1)$ to $(2n,0)$ are
equal in number to the symmetric diagrams related to $\BrM(\ddC_{n-1})$.
\end{proof}

The reader may have wondered why the study of symmetric diagrams was
restricted to type $\ddA_m$ for $m$ odd. The answer is that, if $m=2n$ is
even, each symmetric diagram has a fixed vertical strand from the dot
$(n,1)$ to $(n,0)$. The removal of this strand leads to an isomorphism of
the algebra of the symmetric diagrams with $\SBr(\ddA_{2n-1})$, and so this
construction provides no new algebra. This is remarkable in that the root
system obtained by projecting $\Phi$ onto the $\sigma$-fixed subspace of the
reflection representation, as in Definition \ref{df:fp}, leads to a root
system of type $\ddB_n$ instead of $\ddC_n$.

On the other hand, the presentation by generators and relations given in
Definition \ref{0.1} suggests, at least when $\delta = 1$, the definition of
the Brauer algebra of type $\ddB_n$ as for $\ddC_n$, but with the roles of
$0$ and $1$ reversed in defining relations
(\ref{0.1.11})--(\ref{0.1.18}). The dimensions for the Brauer algebras
$\Br(\ddB_n,R,1)$ with $n\le5$ thus obtained were found to be

\np
\begin{center}
\begin{tabular}{c|c|c|c|c|c}
   $n$&  1 & 2 & 3 & 4& 5   \\
  \hline
  $a_n$& 3 & 25& 273 &3801 & 66315 \\
\end{tabular}
\end{center}

\np
It is likely that these algebras emerge with a construction similar to the
one given for $\ddC_n$ but with $\ddA_{2n-1}$ replaced by $\ddD_{n+1}$ and
$\sigma$ by the diagram automorphism interchanging the two short end nodes.
This will be the subject of further investigation.

At the time of writing of this paper, Chen \cite{ZhiChen} presented a
definition of a generalized Brauer algebra of type $\ddI_2(m)$.  For $m=4$,
this type coincides with $\ddC_2$, but Chen's algebra has dimension $2m+m^2 =
24$, whereas our $\Br(\ddC_2)$ has dimension 25.

\end{section}

\begin{section}{Acknowledgments}
The authors thank David Wales for his valuable comments and helpful suggestions during the preparation of this manuscript.
Max Horn performed computations in GAP to explore the first few algebras
of the series and their actions on admissible sets. We are very grateful for
his help and enthusiastic support. 
\end{section}

\end{document}